\documentclass[graybox]{svmult}

\usepackage{mathptmx}
\usepackage{helvet}
\usepackage{courier}

\usepackage{enumerate,amssymb,amsmath,graphics,epsfig}
\usepackage{color}
\usepackage[nosepfour,warning,np,debug,autolanguage]{numprint}
\newcommand\RE{\mathbb{R}}

\renewcommand\I{\mathbb{I}}
\renewcommand\div{\mathop{\rm{div}}\nolimits}
\renewcommand\det{\mathop{\rm{det}}\nolimits}

\newcommand\Div{\mathop{\mathbf{Div}}\nolimits}
\newcommand\Grad{\mathop{\boldsymbol\nabla}\nolimits}

\newcommand\grad{\mathop\nabla\nolimits}
\newcommand\grads{\mathop{\nabla_s}\nolimits}
\renewcommand\D{\mathbf{D}}

\newcommand\Huod{H^1_{0,D}(\Omega)^d}
\newcommand\Ldo{L^2_0(\Omega)}
\newcommand\Hub{H^1(\B)^d}
\newcommand\Hubd{(H^1(\B)^d)'}
\newcommand\LdB{L^2(\B)}
\newcommand\LdO{L^2(\Omega)}

\newcommand\Oft{\Omega^f_t}
\newcommand\Ost{\Omega^s_t}
\newcommand\B{\mathcal B}

\renewcommand\u{\mathbf{u}}
\renewcommand\v{\mathbf{v}}
\renewcommand\d{\mathbf{d}}
\renewcommand\b{\mathbf{b}}
\renewcommand\c{\mathbf{c}}
\newcommand\f{\mathbf{f}}
\newcommand\n{\mathbf{n}}
\newcommand\w{\mathbf{w}}
\newcommand\x{\mathbf{x}}
\newcommand\z{\mathbf{z}}
\newcommand\s{\mathbf{s}}
\newcommand\X{\mathbf{X}}
\newcommand\Y{\mathbf{z}}

\newcommand\LL{\boldsymbol{\Lambda}}
\newcommand\ssigma{\boldsymbol\sigma}
\newcommand\llambda{\boldsymbol\lambda}
\newcommand\mmu{\boldsymbol\mu}
\newcommand\ttau{\boldsymbol\tau}
\newcommand\F{\mathbb{F}}
\newcommand\C{\mathbb{C}}
\renewcommand\P{\mathbb{P}}
\newcommand\Pe{\P^e_s}
\newcommand\ds{\mathrm{d}\s}
\newcommand\dx{\mathrm{d}\x}
\newcommand\da{\mathrm{da}}

\newcommand\dt{\Delta t}
\newcommand\ucX{\u(\X(\cdot,t),t)}
\newcommand\vcX{\v(\X(t))}

\newcommand\ucXn{\u^{n+1}(\X^n)}
\newcommand\vcXn{\v(\X^n)}
\newcommand\qcXn{q(\X^n)}
\newcommand\pcXn{p^{n+1}(\X^n)}
\newcommand\uhcXn{\u^{n+1}(\X^n)}

\newcommand\phcXn{p^{n+1}(\X^n)}

\newcommand\rhos{\rho_{s_0}}
\newcommand\dr{\delta_\rho}

\newcommand\dwt{\dot\w}
\newcommand{\Af}{A_f}
\newcommand{\As}{A_s}
\newcommand{\Mf}{M_f}
\newcommand{\Ms}{M_s}
\newcommand{\Bf}{B_f}
\newcommand{\Bs}{B_s}
\newcommand{\Cs}{C_s}
\newcommand{\Cf}{C_f}
\newcommand{\Vh}{\mathbf{V}_h}
\newcommand{\Qh}{Q_h}
\newcommand{\Sh}{\mathbf{S}_h}
\newcommand{\Lh}{\boldsymbol{\Lambda}_h}
\newtheorem{thm}{Theorem}
\usepackage{subfigure}
\usepackage{colortbl}
\usepackage{pgfplots}
\usepackage{pgfplotstable}
\usepackage{tikz}
\usepackage{tkz-euclide}
\usetkzobj{all}
\usetikzlibrary{decorations.pathreplacing,shapes.misc,calc,intersections,arrows,external}
\usepgfplotslibrary{external}
\tikzsetexternalprefix{./figures/}
\usepackage{booktabs} 

\graphicspath{{./figures/}}

\pgfplotsset{compat=1.14}

\begin{document}

\title*{A distributed Lagrange formulation of the Finite
 Element Immersed Boundary Method for fluids interacting 
with compressible solids.}

\titlerunning{FEIBM --- compressible case}

%\author{Daniele Boffi}
%\address{Dipartimento di Matematica ``F. Casorati'', Universit\`a di Pavia,
%Italy}
%\email{daniele.boffi@unipv.it}
%\urladdr{http://www-dimat.unipv.it/boffi/}
%%
%% \author{Nicola Cavallini}
%% \address{???, Trieste, Italy}
%% \email{nicola.cavallini@sissa.it}
%% \urladdr{http://www-dimat.unipv.it/\~{}cavallini/}
%%
%\author{Lucia Gastaldi}
%\address{DICATAM, Universit\`a di Brescia, Italy}
%\email{lucia.gastaldi@unibs.it}
%\urladdr{http://lucia-gastaldi.unibs.it}
%%
%\author{Luca Heltai}
%\address{SISSA, Trieste, Italy}
%\email{heltai@sissa.it}
%\urladdr{http://people.sissa.it/\~{}heltai}

\author{Daniele Boffi \and Lucia Gastaldi \and Luca Heltai}

\institute{Daniele Boffi \at Dipartimento di Matematica ``F. Casorati'',
Universit\`a di Pavia, Italy and Department of Mathematics and System
Analysis, Aalto University, Finland, \email{daniele.boffi@unipv.it},
\url{http://www-dimat.unipv.it/boffi/} \and Lucia Gastaldi \at DICATAM,
Universit\`a di Brescia, Italy, \email{lucia.gastaldi@unibs.it},
\url{http://lucia-gastaldi.unibs.it} \and Luca Heltai \at SISSA, Trieste,
Italy, \email{heltai@sissa.it}, \url{http://people.sissa.it/\~{}heltai}}

%\subjclass{65M60, 65M12, 65M85}

\maketitle

%\begin{abstract}

\abstract*{We present a distributed Lagrange multiplier formulation of the
  Finite Element Immersed Boundary Method to couple incompressible
  fluids with compressible solids.
  This is a generalization of the formulation presented in Heltai and
  Costanzo (2012), that offers a cleaner variational formulation,
  thanks to the introduction of distributed Lagrange multipliers, that
  acts as intermediary between the fluid and solid equations, keeping
  the two formulation mostly separated.
  Stability estimates and a brief numerical validation are presented.}

\abstract{We present a distributed Lagrange multiplier formulation of the
  Finite Element Immersed Boundary Method to couple incompressible
  fluids with compressible solids.
  This is a generalization of the formulation presented in Heltai and
  Costanzo (2012), that offers a cleaner variational formulation,
  thanks to the introduction of distributed Lagrange multipliers, that
  acts as intermediary between the fluid and solid equations, keeping
  the two formulation mostly separated.
  Stability estimates and a brief numerical validation are presented.}

%\end{abstract}

\section{Introduction}
\label{se:intro}

Fluid-structure interaction (FSI) problems are everywhere in
engineering and biological applications, and are often too complex to
solve analytically.

Well established techniques, like the Arbitrary Lagrangian Eulerian
(ALE) framework~\cite{HughesLiuK.-1981-a}, enable
the numerical simulation of FSI problems by coupling computational
fluid dynamics (CFD) and computational structural dynamics (CSD),
through the introduction of a deformable fluid grid, whose movement at
the interface is driven by the coupling with the CSD simulation, while
the interior is deformed arbitrarily, according to some smooth
deformation operator.

Although this technique has reached a great level of robustness,
whenever changes of topologies are present in the physics of the
problem, or when freely floating objects (possibly rotating) are
considered, a deforming fluid grid that follows the solid may no
longer be a feasible solution strategy.

The Immersed Boundary Method (IBM), introduced by Peskin in the
seventies~\cite{Peskin-1977-a} to simulate the interaction
of blood flow with heart valves, addressed this issue by reformulating
the coupled FSI problem as a ``reinforced fluid'' problem, where the
CFD system is solved everywhere (including in the regions occupied by
the solid), and the presence of the solid is taken into account in the
fluid as a (singular) source term
(see~\cite{Peskin-2002-a} for a review).

In the original IBM, the body forces expressing the FSI are determined
by modeling the solid body as a network of elastic fibers with a
contractile element, where each point of the fiber acts as a singular
force field (a Dirac delta distribution) on the fluid.

Finite element variants of the IBM were first proposed, almost
simultaneously, by \cite{BoffiGastaldi-2003-a},
\cite{WangLiu-2004-a}, and
\cite{ZhangGerstenbergerWang-2004-a}. However,
only~\cite{BoffiGastaldi-2003-a} exploited the variational
definition of the Dirac delta distribution directly in the finite
element (FE) approximations. 

Such approximation was later generalized to thick hyper-elastic bodies
(as opposed to fibers)~\cite{BoffiGastaldiHeltaiPeskin-2008-a}, where
the constitutive behavior of the immersed solid is assumed to be
incompressible and visco-elastic with the viscous component of the
solid stress response being identical to that of the fluid.

In~\cite{HeltaiCostanzo-2012-a}, the authors present a formulation that
is applicable to problems with immersed bodies of general topological
and constitutive characteristics, without the use of Dirac delta
distributions, and with interpolation operators between the fluid and
the solid discrete spaces that guarantee semi-discrete stability
estimates and strong consistency. Such formulation has been
successfully used~\cite{RoyHeltaiCostanzo-2015-a} to match standard
benchmark tests~\cite{TurekHron-2006-a}.

In this work we show how the incompressible version of the FSI model
presented in~\cite{HeltaiCostanzo-2012-a} can be seen as a special case
of the Distributed Lagrange Multiplier method, introduced in
\cite{BoffiCavalliniGastaldi-2015-a}, and we present a novel
distributed Lagrange multiplier method that generalizes the
compressible model introduced in~\cite{HeltaiCostanzo-2012-a}.

We provide a general variational framework for Immersed Finite Element
Methods (IFEM) based on the distributed Lagrange multiplier
formulation that is suitable for general fluid structure interaction
problems.

\section{Setting of the problem}
Let $\Omega\subset\RE^d$, with $d=2,3$, be a fixed open bounded polyhedral
domain with Lipschitz boundary which is split into two time dependent
subdomains $\Oft$ and $\Ost$, representing the fluid and the solid regions,
respectively. Hence $\Omega$ is the interior of 
$\overline{\Omega}^f_t\cup\overline{\Omega}^s_t$ and we denote by 
$\Gamma_t=\overline{\Omega}^f_t\cap\overline{\Omega}^s_t$ the moving interface
between the fluid and the solid regions. For simplicity, we assume that
$\Gamma_t\cap\partial\Omega=\emptyset$.

\begin{figure}
  \label{fig:domain}
  \centering
  \tikzsetnextfilename{domain}
  \begin{tikzpicture}
    % Some references points
    \node at (-5,1)       (B0){};     
    \node at (-3,1)       (B1){};
    \node at (-3,4)       (B2){};
    \node at (-5,4)       (B3){};
    \node at (-4.8,2.5) (B4){};
    
    % The reference body
    \draw [fill=lightgray] plot [smooth cycle, tension=1]
    coordinates { (B0) (B1) (B2) (B3) (B4) };

    % The box domain
    \draw (0,0) rectangle (5,5) node[anchor=south west](O) {$\Omega$};
    
    % Names G and N to position \Gamma and \nu
    \node at (1,1) (G){};
    \node at (2,4) (N){};
    \node at (1.5,4.5) (Np){};
    
    % The body
    \draw [fill=lightgray] plot [smooth cycle, tension=1]
    coordinates { (G) (3,1)  (3.5,2.3)  (N)};

    % The reference coordinate s and X
    \node at (-4,3) (s) {};
    \node at (2,2) (X) {}; 

    % The letters
    \node at (4,4) {$\Oft$};
    \node at (2.2,3.5) {$\Ost$};
    \node [below left] at (G) {$\Gamma_t$};
    \node [below left] at (B0) {$\B$};
    \node [above] at (s) {$\s$};
    \node [above right] at (X) {$\X(\s,t)$};
    % \node [above] at (N) {$\nu$};

    \draw  (s) node[fill,circle,inner sep=0pt,minimum size=3pt] {};
    \draw  (X) node[fill,circle,inner sep=0pt,minimum size=3pt] {};

    % Arrow
    \path[->] (s) edge
    [bend left] node[above ] {} (X); 
    
    \node at (-1.5,4) {$\X:=\s + \w(\s,t)$};
    
    % The normal
    %\draw [-stealth] (N.south east) -- (Np);
    %\fill (N) circle (1pt);
  \end{tikzpicture}

  \caption{Geometrical configuration of the FSI problem}
\end{figure}
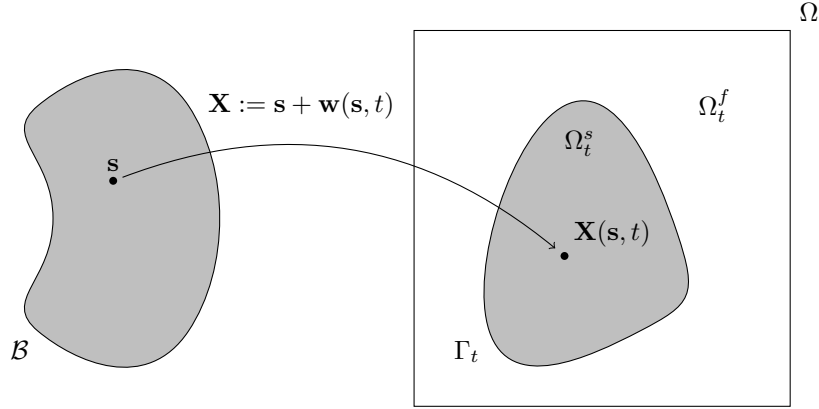

The current position of the solid $\Ost$ is the image of a reference domain
$\B$ through a mapping $\X:\B\to\Ost$. The displacement of the solid is
indicated by $\w$ and for any point $\x\in\Ost$ we have
$\x=\X(\s,t)=\X_0(\s)+\w(\s,t)$ for some $\s\in\B$, where $\X_0:\B\to\Omega_0^s$
denotes the mapping providing the initial configuration of $\Ost$.
For convenience, we assume that the reference domain $\B$ coincides with
$\Omega_0^s$ so that $\X(s,t)=\s+\w(\s,t)$ and $\B\subset\Omega$.
From the above definitions, $\F=\grads\X=\I+\grads\w$ stands for the
deformation gradient and $J=\det(\F)$ for its Jacobian.

We denote by $\u_f:\Omega\to\RE^d$ and $p_f:\Omega\to\RE$
the fluid velocity and pressure and 
assume that the solid velocity $\u_s$ is equal to the material velocity
of the solid, that is
\begin{equation}
\label{eq:constraint}
\u_s(\x,t)=\frac{\partial\X(\s,t)}{\partial t}\Big|_{\x=\X(\s,t)}
=\frac{\partial\w(\s,t)}{\partial t}\Big|_{\x=\X(\s,t)}.
\end{equation}
We indicate with generic symbols $\u, p, \rho$ the \emph{Eulerian}
fields, depending on $\x$ and $t$, which describe the velocity, pressure, and
density, respectively, of a material particle (be it solid or fluid).
By $\dot\u$ we denote the material derivative of $\u$, which
in Eulerian coordinates is expressed by
\begin{equation}
\label{eq:material}
\dot\u(\x,t)=\frac{\partial \u}{\partial t}(\x,t)+\u(\x,t)\cdot\Grad\u(\x,t).
\end{equation}
In the Lagrangian framework, the material derivative coincides with the partial
derivative with respect to time, so that
\[
\dot\w(\s,t)=\partial\w(\s,t)/\partial t.
\]
Continuum mechanics models are based on the conservation of three main
properties: linear momentum, angular momentum, and mass.

When expressed in Lagrangian coordinates, mass conservation is
guaranteed if the reference mass density $\rho_0$ is time
independent. If expressed in Eulerian coordinates, however, mass
conservation takes the form
\begin{equation}
  \label{eq:mass-conservation}
  \dot \rho +  \rho \div \u = 0 \quad \text{in }\Omega,
\end{equation}
and should be included in the system's equations.

Conservation of both momenta can be expressed, in \emph{Eulerian}
coordinates, as:
\begin{equation}
\label{eq:conservation-momenta}
\rho\dot\u =\div\ssigma+\rho\b \quad \text{in }\Omega,
\end{equation}
where $\rho$ is the mass density distribution, $\u$ the velocity, 
$\ssigma$ is the Cauchy stress tensor (its symmetry implies
that conservation of angular momenta is guaranteed by
equation~\eqref{eq:conservation-momenta}), and $\b$ describes the
external force density per unit mass acting on the system. Such a
description is common to all continuum mechanics models (see, for
example,~\cite{Gurtin-1982-a}). The equations for fluids and solids
are different according to their constitutive behavior, i.e.,
according to how $\ssigma$ relates to $\u, \w$, or $p$.

%In the fluid and in the solid the mass
%density distribution has different values and the Cauchy stress tensor 
%has distinct constitutive equations. Therefore we set:
%%
%\[
%\rho(\x,t)=
%\begin{cases}
%\rho_f & \text{in }\Oft \\
%\rho_s(\x,t) & \text{in }\Ost.
%\end{cases}
%\]
%%
%We 
%
If the material is incompressible, it can be shown that $\div \u = 0$
everywhere, and the material derivative of the density is constantly
equal to zero (from equation~\eqref{eq:mass-conservation}). Notice
that this does not imply that $\rho$ is constant (neither in time nor
in space), and it is still in general necessary to include
equation~\eqref{eq:mass-conservation} in the system.

For incompressible materials, however, the volumetric part of the
stress tensor can be interpreted as a Lagrange multiplier associated
with the incompressibility constraints. For incompressible fluids, the
stress is decomposed into $\ssigma_f=-p_f\I+\nu_f\D\u_f$ where $\u_f$
is the fluid velocity, $p_f$ the pressure, $\nu_f>0$ is the viscosity
coefficient and $\D\u_f=(1/2)\left(\Grad\u_f+(\Grad\u_f)^\top\right)$.

Hence the equations describing the fluid motion are the well-known
Navier--Stokes equations, that is:
\begin{equation}
\label{eq:fluideq}
\aligned
&\rho_f\dot\u_f
-\div(\nu_f\D\u_f)+\grad p_f=\rho_f\b &&\text{in }\Oft\\
&\div\u_f=0&&\text{in }\Oft,
\endaligned
\end{equation}
where we assumed that $\rho_f$ is constant throughout $\Oft$.

As far as the solid is concerned, we assume that it is composed by a
viscous elastic material so that the Cauchy stress tensor can be
decomposed into the sum of two contributions: a viscous part and a
pure elastic part, as follows
\begin{equation}
\label{eq:solidsigma}
\ssigma_s=\ssigma_s^v+\ssigma_s^e=\nu_s\D\u_s+\ssigma_s^e.
\end{equation}
Here $\u_s$ is the solid velocity, $\nu_s\ge0$ is the solid viscosity
coefficient and $\ssigma_s^e$ denotes the elastic part of the solid
Cauchy stress tensor. We assume this elastic part to behave
hyper-elastically, i.e., we assume that there exists an elastic
potential energy density $W(\F)$ such that $W(\mathbf{R} \F) = W(\F)$
for any rotation $\mathbf{R}$, that represents the amount of elastic
energy stored in the current solid configuration, and that depends
only on its deformation gradient $\F = \grads \X$.

When expressed in Lagrangian coordinates, a possible measure for the
elastic part of the stress is the so called first Piola--Kirchhoff
stress tensor, defined as the Fr\'echet derivative of $W$ w.r.t.\ to
$\F$, i.e.:
\begin{equation}
  \label{eq:first-piola-kirchhoff}
  \P_s^e := \frac{\partial W}{\partial \F}.
\end{equation}

The first Piola--Kirchhoff stress tensor allows one to express the
conservation of linear momentum in Lagrangian coordinates as
\begin{equation}
  \label{eq:conservation-momentum-lagrangian}
  \rhos \ddot \w = \Div(\P) + \rhos \mathbf{B} \quad \text{ in } \B,
\end{equation}
where, similarly to its Eulerian counterpart, we assume that $\P$ is
decomposed in an \emph{additive} way into its viscous part $\P_s^v$
and into its elastic part $\P_s^e$, defined in
equation~\eqref{eq:first-piola-kirchhoff}.

For any portion $\mathcal{P} \subset \B$ of the solid (with outer
normal $\mathbf{N}$) deformed to $\mathcal{P}_t$ (with outer normal
$\mathbf{n}$), the following relation between $\P$ and $\ssigma$ holds:
\begin{equation}
  \label{eq:integral-relation-stresses}
  \int_{\mathcal{P}} \P \mathbf{N} \, \d \Gamma_s =
  \int_{\mathcal{P}_t} \ssigma \n \,\d \Gamma_x, \qquad \forall
  \mathcal{P} \subset \B, \qquad \mathcal{P}_t := \X(\mathcal{P}, t),
\end{equation}
that is, we can express pointwise the \emph{viscous} part of the solid
stress in Lagrangian coordinates by rewriting the first
Piola--Kirchhoff stress $\P^v_s$ in terms of
$\ssigma_s^v := \nu_s \D\u_s$, and the \emph{hyper-elastic} part of
the solid stress in Eulerian coordinates by expressing the Cauchy
stress $\ssigma_s^e$ in terms of $\P_s^e := \partial W/\partial \F$:
\begin{equation}
\label{eq:PiolaKirchhoff}
\aligned
\P^v_s(\s,t) &
:=J^{\phantom{-1}}\ssigma_s^v(\x,t)\F^{-\top}(\s,t)&&\quad\text{for }\x=\X(\s,t) \\
\ssigma^e_s(\x,t) &
:=J^{-1}\P_s^e(\s,t)\ \F^{\top }(\s,t)\phantom{-}&&\quad\text{for }\x=\X(\s,t).
\endaligned
\end{equation}

With these definitions, the conservation of linear momentum for the
solid equation can be expressed either in Lagrangian coordinates as
\begin{equation}
  \label{eq:solideq-lagrangian}
  \rhos \ddot \w = \Div \P^v_s+\Div \frac{\partial W}{\partial \F} +
  \rhos\mathbf{B} \quad \text{ in } \B 
\end{equation}
or in Eulerian coordinates as
\begin{equation}
\label{eq:solideq}
\rho_s\dot\u_s-\div(\nu_s\D\u_s)-\div\ssigma_s^e=\rho_s\b\quad \text{in
}\Ost.
\end{equation}

Notice that the conservation of mass for the solid equation is a
simple kinematic identity that derives from the fact that $\rhos$ does
not depend on time, i.e., we have
\begin{equation}
  \label{eq:solid-mass-conservation-eulerian}
 \frac{\dot \rho_s}{\rho_s} +  \div \u_s = 0  \qquad \text{in
 }\Ost,
\end{equation}
or, equivalently,
\begin{equation}
\label{eq:solidmass}
\rho_s(\x,t)=\rhos(\s)/J(\s,t)\quad\text{for }\x=\X(\s,t),
\end{equation}
that is:
 \qquad
 \begin{equation}
   \label{eq:alternative-solidmass}
   \div \u_s(\x,t) = \frac{\dot J}{J}(\s,t) \quad
   \quad\text{for }\x=\X(\s,t).
\end{equation}

%
% Hence the equation in the solid has the following form:
%
%

The equations in the solid and in the fluid are coupled through interface
conditions along $\Gamma_t$, which enforce the continuity of the velocity,
corresponding to the no-slip condition between solid and fluid, and the balance
of the normal stress:
\begin{equation}
\label{eq:interface}
\aligned
&\u_f=\u_s &&\text{on }\Gamma_t\\
&\ssigma_f\n_f+\ssigma_s\n_s=0&&\text{on }\Gamma_t,
\endaligned
\end{equation}
where $\n_f$ and $\n_s$ denote the outward unit normal vector to $\Oft$ and
$\Ost$, respectively.

The system is complemented with initial and boundary conditions. The boundary
$\partial\Omega$ is split into two parts
$\partial\Omega_D$ and $\partial\Omega_N$, where Dirichlet and Neumann
conditions are imposed, respectively, with
$\partial\Omega_D\cap\partial\Omega_N=\emptyset$. Since we assumed that
$\partial\Omega\cap\Gamma_t=\emptyset$, the initial and boundary conditions are
given by:  
\begin{equation}
\label{eq:initialandbc}
\aligned
&\u_f(0)=\u_{f0}&&\text{in }\Omega_0^f\\
&\u_s(0)=\u_{s0}&&\text{in }\Omega_0^s\\
&\X(0)=\X_0&&\text{in }\B\\
&\u_f=\u_g&&\text{on }\partial\Omega_D\\
&(\nu_f\D\u_f-p_f\I)\n_f=\ttau_g&&\text{on }\partial\Omega_N.
\endaligned
\end{equation}
In the following, we shall consider $\u_g=0$ on $\partial\Omega_D$,
for simplicity, and we define the space $\Huod$ as the space of
functions in $H^1(\Omega)^d$ such that their trace on
$\partial\Omega_D$ is zero.

By multiplying by $\v\in\Huod$ the first equation in~\eqref{eq:fluideq}
and~\eqref{eq:solideq}, integrating by part and using the second
interface condition in~\eqref{eq:interface}, we arrive at the
following weak form of the fluid-structure interaction problem: find
$\u_f$, $p_f$, $\u_s$, and $\w$ such that~\eqref{eq:constraint},
\eqref{eq:interface}$_1$ and~\eqref{eq:initialandbc} are satisfied and it holds:
\begin{equation}
\label{eq:firstweak}
\aligned
\int_{\Oft}&\rho_f(\dot\u_f-\b)\v\,\dx+\int_{\Oft}\nu_f\D\u_f:\D\v\,\dx
-\int_{\Oft} p_f\div\v\,\dx\\
&+\int_{\Ost}\rho_s(\dot\u_s-\b)\v\,\dx+\int_{\Ost}\nu_s\D\u_s:\D\v\,\dx
+\int_{\Ost} \ssigma_s^e:\D\v\,\dx\\
&\qquad\qquad=\int_{\partial\Omega_N}\ttau_g\cdot\v\,\da\qquad\forall\v\in\Huod\\
\int_{\Oft}&\div\u_f q_f\,\dx=0\qquad\qquad\qquad\forall q_f\in L^2(\Oft)
\endaligned
\end{equation}
where the notation $\dot\u$ indicates the material derivative with
respect to time. 
Using a \emph{fictitious domain} approach, we transform
problem~\eqref{eq:firstweak} by introducing the following new unknowns. 
Thanks to the continuity condition for fluid and solid velocity, we
define $\u\in\Huod$ and $p \in \LdO$ as 
\begin{equation}
\label{eq:ufictitious}
\u=\left\{
\begin{array}{ll}
\u_f &\text{in }\Oft\\
\u_s &\text{in }\Ost.
\end{array}
\right., \qquad p =\left\{
\begin{array}{ll}
p_f &\text{in }\Oft\\
p_s = 0 &\text{in }\Ost.
\end{array}
\right..
\end{equation} 
Notice that the pressure field does not have any physical meaning in
the solid, and it is (weakly) imposed to be zero.

Then we can write:
\begin{equation}
\label{eq:momentum}
\aligned
\int_\Omega&\rho_f\dot\u\v\,\dx+\int_\Omega\nu_f\D(\u):\D(\v)\,\dx
-\int_\Omega p\div\v\,\dx 
+\int_{\Ost}(\rho_s-\rho_f)\dot\u\v\,\dx \\
&+\int_{\Ost}(\nu_s-\nu_f)\D\u:\D\v\,\dx
+\int_{\Ost} \ssigma_s^e:\D\v\,\dx+\int_{\Ost} p\div\v\,\dx\\
&\qquad\qquad=\int_\Omega\f\v\,\dx +
\int_{\partial\Omega_N}\ttau_g\cdot\v\,\da
&&\hspace{-1cm}\forall\v\in\Huod\\
\int_\Omega&\div\u\, q\,\dx-\int_{\Ost}\div\u\, q\,\dx +
\int_{\Ost} \frac1\kappa p q\,\dx=0&&
\hspace{-1cm}\forall q\in \LdO
\endaligned
\end{equation}
where
\[
\f=\left\{\begin{array}{ll}
\rho_f\b &\text{in }\Oft\\
\rho_s\b &\text{in }\Ost
\end{array}
\right.,
\]
and $\kappa$ plays the role of a bulk modulus constant.

In the solid the Lagrangian framework should be preferred, hence we transform
the integrals over $\Ost$ into integrals on the reference domain $\B$.
Recalling~\eqref{eq:constraint} and~\eqref{eq:solidmass}
the equations in~\eqref{eq:momentum} are rewritten in the following form:
\begin{equation}
\label{eq:momentum2}
\aligned
\int_\Omega&\rho_f\dot\u(t)\v\,\dx+\int_\Omega\nu_f\D\u(t):\D\v\,\dx
-\int_\Omega p\div\v\,\dx\\
&
+\int_{\B}(\rhos-\rho_fJ)\ddot\w(t)\v(\X(\s,t))\,\ds
+V\Big(\dot\w(t),\v(\X(\s,t))\Big)\\
&+\int_{\B} \Pe(t):\Grad_s\v(\X(\s,t))\,\ds
+\int_{\B} Jp(\X(\s,t),t)\F^{-\top}:\nabla_s\v(\X(\s,t))\,\ds\\
&\qquad\qquad=\int_\Omega\f(t)\v\,\dx+
\int_{\partial\Omega_N}\ttau_g\cdot\v\,\da
&&\hspace{-1cm}\forall\v\in\Huod\\
\int_\Omega&\div\u(t)\, q\,\dx-
\int_{\B} J\,q(\X(\s,t))\F^{-\top}:\Grad_s\dot\w(t)\, \ds\\
&\qquad\qquad+\int_\B\frac1\kappa p(t) qJ\,\dx=0
&&\hspace{-1cm}\forall q\in \LdO
\endaligned
\end{equation}
where, for all $\X,\Y\in\Hub$
\begin{equation}
\label{eq:defV}
V(\X,\Y)=
\frac{\nu_s-\nu_f}4
\int_\B(\Grad_s\X\F^{-1}+\F^{-\top}\Grad_s\X^\top):
(\Grad_s\Y\F^{-1}+\F^{-\top}\Grad_s\Y^\top)J\,\ds.
\end{equation}

Since $\X(t):\B\to\Ost$ is one to one and belongs to $W^{1,\infty}(\B)^d$, 
$\Y=\vcX$ is an arbitrary element of $\Hub$ when $\v$ varies in $\Huod$.

Let $\LL$ be a functional space to be defined later on and 
$\c:\LL\times\Hub\to\RE$ a bilinear form such that 
\begin{equation}
\label{eq:defc}
\aligned
&\c \text{ is continuous on } \LL\times\Hub\\
&\c(\mmu,\z)=0 \text{ for all } \mmu\in\LL \text{ implies } \z=0.
\endaligned
\end{equation}
For example, we can take $\LL$ as the dual space of $\Hub$ and define $\c$ as
the duality pairing between $\Hub$ and $\Hubd$, that is:
\begin{equation}
\label{eq:defc1}
\c(\mmu,\Y)=\langle\mmu,\Y\rangle\quad\forall\mmu\in\Hubd,\ \Y\in\Hub.
\end{equation}
Alternatively, one can set $\LL=\Hub$ and define
\begin{equation}
\label{eq:defc2}
\c(\mmu,\Y)=\left(\Grad_s\mmu,\Grad_s\Y\right)_\B+\left(\mmu,\Y\right)_\B
\quad\forall\mmu,\ \Y\in\Hub.
\end{equation}
With the above definition for $\c$, we introduce an unknown $\llambda\in\LL$
such that 
\begin{equation}
\label{eq:Lagrange}
\aligned
\c(\llambda,\Y)&=\int_{\B}(\rhos-\rho_fJ)\ddot\w(t)\Y)\,\ds
+V(\dot\w(t),\v(\X(\s,t)))\\
&+\int_{\B} \P(\F(t)):\Grad_s\Y\,\ds
+\int_{\B} Jp(\X(\s,t),t)\F^{-\top}:\Grad_s\Y\,\ds\qquad\forall\Y\in\Hub.
\endaligned
\end{equation}
Hence we can write the following problem.
\begin{problem}
\label{pb:1}
Let us assume that
$\u_0\in\Huod$, $\X_0\in W^{1,\infty}(\B)^d$, and that for all $t\in[0,T]$ 
$\ttau_g(t)\in H^{-1/2}(\partial\Omega_N)$, and $\f(t)\in\LdO$. 
For almost every $t\in]0,T]$, find
$(\u(t),p(t))\in H^1(\Omega)\times\LdO$, $\w(t)\in\Hub$,
and $\llambda(t)\in\LL$ such that it holds
\begin{subequations}
\begin{alignat}{2}
  &\rho_f(\dot\u(t),\v)+a(\u(t),\v)-(\div\v,p(t)) \ \notag&&\\
  &\quad+\c(\llambda(t),\vcX)
=(\f(t),\v)+(\ttau_g(t),\v)_{\partial\Omega_N}
   &&\ \forall\v\in\Huod
     \label{eq:NS1_DLM}\\
  &-(\div\u(t),q)+(J\,q(\X(\s,t))\F^{-\top},\Grad_s\dot\w(t))_\B
\notag&&\\
&\quad-\frac1\kappa (Jp(t),q)_\B=0&&\ \forall q\in\LdO
     \label{eq:NS2_DLM}\\
  &\left(\dr\ddot\w(t),\Y\right)_{\B}
+(\Pe(t),\Grad_s\Y)_{\B}+V(\dot\w(t),\Y)
\ \notag&&\\
&\quad+(Jp(\X(\s,t),t)\F^{-\top},\Grad_s\Y)_\B-\c(\llambda(t),\Y)=0&&
\ \forall\Y\in\Hub
     \label{eq:solid_DLM}\\
  &\c\left(\mmu,\ucX-\dwt(t)\right)
   =0 &&\ \forall\mmu\in \LL
     \label{eq:vincolo1}\\
& \X(\s,t)=\s+\w(\s,t)\quad\text{\rm for }\s\in\B\\
&\u(0)=\u_0\quad \mbox{\rm in }\Omega,\qquad\X(0)=\X_0\quad \mbox{\rm in }\B.
     \label{eq:ci_DLM}
\end{alignat}
\end{subequations}
\end{problem}
Here $\dr=\rhos-\rho_fJ$, $(\cdot,\cdot)$ and $(\cdot,\cdot)_\B$ stand for
the scalar product in $\LdO$ and $\LdB$, respectively, and
\[
\aligned
&a(\u,\v)=(\nu_f\D\u,\D\v) \quad\forall\u,\v\in\Huod\\
&(\ttau_g,\v)_{\partial\Omega_N}=\int_{\partial\Omega_N}\ttau_g\cdot\v\,\da\\
&(\P,\mathbb{Q})_\B=\int_\B \P:\mathbb{Q}\,\ds\quad \text{for }\P,\mathbb{Q}
\text{ tensors in } \LdB.
\endaligned
\]
\begin{proposition}
\label{pr:press_s}
Let $(\u,p,\w,\llambda)$ be a solution of
Problem~\ref{pb:1}. We have that $p(t)=0$ in $\Ost$ for $t\in]0,T]$ and
$(\div\u,q)_{\Oft}=0$ for all $q\in L^2(\Oft)$.
\end{proposition}
\begin{proof}
The constraint
in~\eqref{eq:vincolo1} together with~\eqref{eq:defc} implies that
$\u(t)=\dot\w(t)$
in $\Ost$. Using this fact and changing variable in the last two integrals
in~\eqref{eq:NS2_DLM}, we arrive at
\[
-(\div\u(t),q)+\int_{\Ost} \div\u(t) q\,\dx
-\int_{\Ost}\frac1\kappa p(t) q\,\dx=0\ \forall q\in\LdO.
\]
Taking $q=p(t)$ in $\Ost$ and vanishing in $\Oft$, we end up with
\[
\int_{\Ost}\frac1\kappa p^2(t)\,\dx=0,
\]
which implies that $p(t)=0$ in $\Ost$.
Taking $q=0$ in $\Oft$ in~\eqref{eq:NS2_DLM} $\Ost$ we obtain that the velocity
is divergence free in the fluid domain. 
\end{proof}

The following theorem gives the estimate of the energy.
\begin{thm} Let $(\u,p,\w,\lambda)$ be the 
solution of Problem~\ref{pb:1}, then the following estimate holds true
\begin{equation}
\aligned
&\frac12\frac d{dt}\|\rho^{1/2}\u(t)\|^2_{0,\Omega}
+\|\nu^{1/2}\D\u(t)\|^2_{0,\Omega}
+\frac d{dt}\int_B W(\F(t))\,\ds\\
&\qquad \le C\left(\|\f(t)\|^2_{0,\Omega}+
\|\ttau_g(t)\|_{H^{-1/2}(\partial\Omega_N)}\right).
\endaligned
\end{equation}
Here 
\[
\rho=\left\{
\begin{array}{ll}
\rho_f &\quad\text{\rm in }\Oft\\
\rho_s(\x,t)&\quad\text{\rm in }\Ost
\end{array}
\right.,
\qquad
\nu=\left\{
\begin{array}{ll}
\nu_f &\quad\text{\rm in }\Oft\\
\nu_s &\quad\text{\rm in }\Ost
\end{array}
\right..
\]
\end{thm}
\begin{proof}
We take the following test functions in the equations listed in
Problem~\ref{pb:1}: $\v=\u(t)$, $q=-p(t)$, $\Y=\partial\w(t)/\partial t$,
and $\mmu=-\llambda(t)$. Summing up all the
equations, and taking into account Proposition~\ref{pr:press_s} and
the constraint in~\eqref{eq:vincolo1} together with~\eqref{eq:defc}, we have
\begin{equation}
\label{eq:p1}
\aligned
\rho_f&\int_\Omega \dot\u(t)\cdot\u(t)\,\dx+\nu_f\|\D\u(t)\|^2_{0,\Omega}
+(\nu_s-\nu_f)\|\D\u\|_{0,\Ost}^2\\
&\qquad+\frac12 \int_B\dr\frac{\partial}{\partial t}(\dwt(t))^2\,\ds
+\Big(\Pe(t),\Grad_s\dwt(t)\Big)_\B\\
&\qquad =(\f(t),\u(t))+
(\ttau_g(t),\u(t))_{\partial\Omega_N}.
\endaligned
\end{equation}
Using again the constraint~\eqref{eq:vincolo1}, we can deal with the
first integral on the second line as follows:
\[
\aligned
\frac12 &\int_B\dr\frac{\partial}{\partial t}(\dwt)^2\,\ds=
\frac12 \int_B\rhos\frac{\partial}{\partial t}(\dwt)^2\,\ds-
\frac12 \int_B\rho_fJ\frac{\partial}{\partial t}(\dwt)^2\,\ds\\
&=\frac12\frac d{dt}\int_B\rhos(\dwt)^2\,\ds-
\frac12\frac d{dt}\int_\B\rho_f J\dwt\,\ds
+\frac12\int_B\rho_f\frac{\partial J}{\partial t}(\dwt)^2\,\ds \\
&=\frac12\frac d{dt}\int_{\Ost}(\rho_s(\x,t)-\rho_f)\u^2(\x,t)\,\dx
+\frac12\int_{\Ost}\rho_f(\div\u)\u^2(\x,t)\,\dx.
\endaligned
\]
We add this relation to the first integral in~\eqref{eq:p1}, and we take into
account the definition of the material derivative; hence we obtain after
integration by parts
\begin{equation}
\label{eq:p2}
\aligned
\rho_f&\int_\Omega \dot\u(t)\cdot\u(t)\,\dx+\frac12
\int_B\dr\frac{\partial}{\partial t}(\dwt(t))^2\,\ds\\
&=\frac{\rho_f}2
\int_\Omega\Big(\frac{\partial\u^2(t)}{\partial t}
+\u(t)\cdot\Grad\u(t)^2\Big)\,\dx
+\frac12\frac d{dt}\int_{\Ost}(\rho_s(\x,t)-\rho_f)\u^2(t)\,\dx\\
&\qquad\qquad+\frac12\int_{\Ost}\rho_f(\div\u(t))\u^2(t)\,\dx\\
&=\frac{\rho_f}2\frac d{dt}\int_\Omega\u^2(t)\,\dx
-\frac{\rho_f}2\int_\Omega(\div\u(t))\u^2(t)\,\dx\\
&\quad+\frac12\frac d{dt}\int_{\Ost}(\rho_s(\x,t)-\rho_f)\u^2(t)\,\dx
+\frac12\int_{\Ost}\rho_f(\div\u(t))\u^2(t)\,\dx\\
&=\frac12\frac d{dt}\int_\Omega\rho\u^2(t)\,\dx
\endaligned
\end{equation} 
Notice that the quantity $1/2\rho\u^2$ represents the kinetic energy density
per unit volume.

From the definition of the deformation gradient, we deduce that
\[
\aligned
&\Big(\Pe(t),\Grad_s\dwt(t)\Big)_\B=
\int_\B\Pe(t):\frac{\partial\F}{\partial t}(t)\,\ds=
\int_\B\frac{\partial W(\F(t))}{\partial\F}:\frac{\partial\F(t)}{\partial t}\,\ds
\\
&\quad =\int_B \frac{\partial W(\F(t))}{\partial t}\,\ds
=\frac d{dt}\int_\B W(\F(t))\,\ds.
\endaligned
\]
The integral on the right hand side represents the elastic energy of the solid.

Putting together these expressions with~\eqref{eq:p1} we obtain the desired
stability estimate.
\end{proof}
\section{Time discretization}
For an integer $N$, let $\dt=T/N$ be the time step, and $t_n=n\dt$ for
$n=0,\dots,N$. We discretize the time derivatives with backward finite
differences and use the following notation:
\[
\partial_t\u^{n+1}=\frac{\u^{n+1}-\u^n}{\dt},
\quad\partial_{tt}\u^{n+1}=\frac{\u^{n+1}-2\u^{n}+\u^{n-1}}{\dt^2}.
\]
By linearization of the nonlinear terms, we arrive at the following
semi-discrete problem:
\begin{problem} 
\label{pb:semidiscrete}
For $n=1,\dots,N$, find $(\u^n,p^n)\in\Huod\times\Ldo$,$\w^n\in\Hub$,
and $\llambda^n\in\LL$ such that it holds
\begin{subequations}
\begin{alignat}{2}
  &\rho_f\Big(\partial_t\u^{n+1},\v)+b(\u^n,\u^{n+1},\v)
+a(\u^{n+1},\v)\ \notag&&\\
  &\quad-(\div\v,p^{n+1})+\c(\llambda^{n+1},\vcXn)\notag&&\\
   &\quad=(\f^{n+1},\v)+(\ttau_g^{n+1},\v)_{\partial\Omega_N}
   &&\ \forall\v\in\Huod
     \label{eq:NS1_sd}\\
  &-(\div\u^{n+1},q)
+(J^n\,\qcXn(\F^n)^{-\top},\Grad_s\partial_t\w^{n+1})_\B
\notag&&\\
&\quad-\frac1\kappa (J^np^{n+1}, q)_\B=0&&\ \forall q\in\LdO
     \label{eq:NS2_sd}\\
  &\left(\dr^n\partial_{tt}\w^{n+1},\Y\right)_{\B}
+({\Pe}^{n+1},\Grad_s\Y)_{\B}+V_n(\partial_t\w^{n+1},\Y)
\ \notag&&\\
&\quad
+(J^n\pcXn(\F^n)^{-\top},\Grad_s\Y)_\B-\c(\llambda^{n+1},\Y)=0&&
\ \forall\Y\in\Hub
     \label{eq:solid_sd}\\
  &\c\left(\mmu,\ucXn-\partial_t\w^{n+1}\right)
   =0 &&\ \forall\mmu\in \LL
     \label{eq:vincolo1_sd}\\
& \X^{n+1}=\s+\w^{n+1}\quad\text{\rm in }\B\\
  &\u^0=\u_0\quad\mbox{\rm in }\Omega,\qquad\X^0=\X_0\quad\mbox{\rm in }\B.
     \label{eq:ci_sd}
\end{alignat}
\end{subequations}
In~\eqref{eq:solid_sd}, $V_n(\X,\Y)$ indicates that the Jacobian $J$
and the deformation gradient $\F$ appearing in the definition of $V$
(see~\eqref{eq:defV}) are computed at time $t_n$, moreover, using the
definition of the first Piola--Kirhhoff stress
tensor~\eqref{eq:first-piola-kirchhoff} we set
$\displaystyle {\Pe}^{n+1}=\frac{\partial W}{\partial\F}(\F^{n+1})$
where different linearizations can be obtained according to the
specific hyper-elastic model in use.
\end{problem}
In~\eqref{eq:NS1_sd} we have $b(\u,\v,\w)=\rho_f(\u\cdot\Grad\v,\w)$.

Problem~\ref{pb:semidiscrete} can be written in operator matrix form as
follows:
\[
\aligned
&\left[
\begin{array}{cc|c|c}
{\Mf}/{\dt}+\Af^n      & \Bf^\top     & 0                & {\Cf^n}^\top  \\
\Bf                    & M_p^n        & \Bs^n            & 0\\
\hline
0                      & {\Bs^n}^\top/\dt   
& {\Ms^n}/{\dt^2}+{\As^{v,n}}/{\dt}+\As^{e,n}  & -\Cs^\top \\
\hline
\Cf^n                  & 0        & -{\Cs}/{\dt} & 0    \\
\end{array}
\right]
\left[\begin{array}{c}
\u^{n+1}\\p^{n+1}\\
\hline
\w^{n+1}\\
\hline
\llambda^{n+1}
\end{array}\right]\\
&\qquad
=\left[\begin{array}{c}
{\Mf}\u^n/{\dt}+\f^{n+1}+\ttau_g^{n+1+\ttau_g^{n+1}}\\
{\Bs^n}^\top\w^n/\dt\\
\hline
{\Ms^n}(2\w^{n}-\w^{n-1})/{\dt^2}+{\As^v}\w^n/{\dt}\\
\hline
{\Cs}\w^n/{\dt}
\end{array}\right]
\endaligned
\]
where
\[
\aligned
&\langle\Mf\u,\v\rangle=(\u,\v),\qquad\langle\Ms^n\w,\z\rangle=(\dr^n\w,\z)_\B,
\qquad\langle M_p^n p,q\rangle=\frac1\kappa(J^n p,q)_\B,\\
&\langle\Af^n\u,\v\rangle=b(\u^n,\u,\v)+a(\u,\v),\\
&\langle\As^{v,n}\w,\z\rangle=V_n(\w,\z),
\qquad \langle\As^{e,n}\w,\z\rangle=
(2\Grad_s\w\cdot\frac{\partial W}{\partial\C}(\F^n),\z)_\B,\\
&\langle\Bf\v,q\rangle=-(\div\v,q),\qquad
\langle\Bs^n\z,q\rangle=(J^n\qcXn(\F^n)^{-\top},\Grad_s\z)_\B,\\
&\langle\Cf^n\v,\mmu\rangle=\c(\mmu,\v(\X^n)),\qquad
\langle\Cs\w,\mmu\rangle=\c(\mmu,\w).
\endaligned
\]
We have used the second Piola-Kirchhoff stress tensor to define the operator
associated to the elastic stress tensor in the solid and $\C=\F^\top\F$.
\section{Space-time discretization}
In this section we introduce the finite element spaces needed for the space-time
discretization of Problem~\ref{pb:1}.
For this we consider two independent meshes in $\Omega$ and in $\B$.
We use a stable pair $\Vh\times\Qh$ of finite elements to discretize fluid
velocity and pressure. We denote by $h$ the maximum  edge size.
For example, we can take a mesh made of simplexes and use 
the Hood--Taylor element of lowest degree or we can subdivide the domain
$\Omega$ in parallelepipeds and apply the $Q_2-P_1$ element. The main
difference among the above finite elements consists in the fact that the
pressure for the Hood--Taylor element is continuous while it is discontinuous in
the $Q_2-P_1$ case. It is well known that discontinuous pressure approximation
enjoy better local mass conservation; possible strategies to improve mass
conservation are presented in~\cite{bcgg2012}.
In the solid, we take a regular mesh of simplexes, where $h_s$ stands for the
maximum edge size and denote by $\Sh$ the finite element space containing
piece wise polynomial continuous functions.
Finally, the finite element space $\Lh$ for the Lagrange multiplier $\llambda$
coincides with $\Sh$.
\begin{problem} 
\label{pb:discrete}
Given $\u_{0,h}\in\Vh$ and $\X_{0,h}\in\Sh$, for $n=1,\dots,N$
find $(\u_h^n,p_h^n)\in\Vh\times\Qh$,$\w_h^n\in\Sh$,
and $\llambda_h^n\in\Lh$ such that it holds
\begin{subequations}
\begin{alignat}{2}
  &\rho_f\Big(\partial_t\u_h^{n+1},\v)+b(\u_h^n,\u_h^{n+1},\v)
+a(\u_h^{n+1},\v)\ \notag&&\\
  &\quad-(\div\v,\phcXn)+\c(\llambda_h^{n+1},\vcXn)\notag&&\\
  &\quad=(\f^{n+1},\v)+(\ttau_g^{n+1},\v)_{\partial\Omega_N}
   &&\ \forall\v\in\Vh
     \label{eq:NS1_h}\\
  &-(\div\u_h^{n+1},q)
+(J^n\,\qcXn(\F^n)^{-\top},\Grad_s\partial_t\w_h^{n+1})_\B
\notag&&\\
&\quad-\frac1\kappa (J^n\phcXn, q)_\B=0&&\ \forall q\in\Qh
     \label{eq:NS2_h}\\
  &\left(\dr^n\partial_{tt}\w^{n+1},\Y\right)_{\B}
+({\Pe}_h^{n+1},\Grad_s\Y)_{\B}+V_{h,n}(\partial_t\w_h^{n+1},\Y)
\ \notag&&\\
&\quad
+(J^n\phcXn(\F^n)^{-\top},\Grad_s\Y)_\B-\c(\llambda_h^{n+1},\Y)=0&&
\ \forall\Y\in\Sh
     \label{eq:solid_h}\\
  &\c\left(\mmu,\uhcXn-\partial_t\w_h^{n+1}\right)
   =0 &&\ \forall\mmu\in \Lh
     \label{eq:vincolo1_h}\\
& \X_h^{n+1}=\s+\w_h^{n+1}\quad\text{\rm in }\B\\
  &\u^0=\u_{0,h}\quad\mbox{\rm in }\Omega,
\qquad\X^0=\X_{0,h}\quad\mbox{\rm in }\B.
     \label{eq:ci_h}
\end{alignat}
\end{subequations}
\end{problem}

\section{Numerical validation}

We have implemented the model described in the previous sections in a
custom C++ code, based on the open source finite element software
library {\tt deal.{I}{I}}~\cite{ArndtBangerthDavydov-2017-a}, and on
a modification of the code presented
in~\cite{HeltaiRoyCostanzo-2014-a}.

We use the classic inf-sup stable pair of finite elements $Q2-P1$, for
the velocity and pressure in the fluid part, and standard $Q2$
elements for the solid part. 

In the following test cases the solid is modeled as a compressible
neo-Hookean material, and the constitutive response function for the
first Piola--Kirchhoff stress of the solid is given as
%%================Equation start=================%%
\begin{equation} \label{eqn:FirstPKStress-CompSolid}
\Pe=\mu^e \left(\F- J^{-2 \nu/(1-2\nu)} \F^{-\mathrm{T}} \right),
\end{equation}
%%================Equation end=================%%
where $\mu^e$ is the shear modulus and $\nu$ is the Poisson's ratio for
the solid.

The tests are designed to validate the correct handling of the
coupling between incompressible fluids and compressible solids. 

\subsection{Recovery and rise of an initially compressed disk in a stationary fluid}
\label{sec:cmame-ball}
%%%
\begin{figure}[htbp]%%------------Figure start ------------%%
\begin{center}
\includegraphics[scale=1.0]{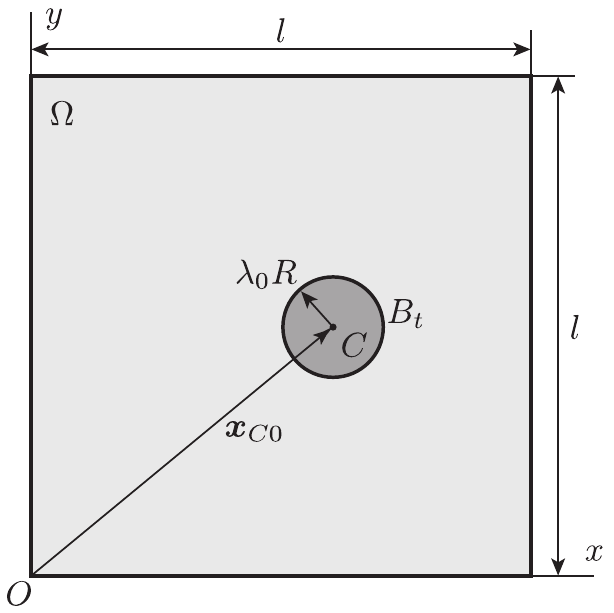}
\caption{Initial configuration of the system with the compressible disk.}
\label{fig:cmame-ball-geometry}
\end{center}
\end{figure}%%------------Figure end ------------%%

This test case has been presented initially in Heltai and
Costanzo~\cite{HeltaiCostanzo-2012-a} and simulates the motion of a
compressible, viscoelastic disk having an undeformed radius $R$. The
disk is initially squeezed (i.e. initial dimension of disk is
$\lambda_{0}R$ with $\lambda_{0}<1$), and left to recover in a square
control volume of edge length $L$ that is filled with an initially
stationary, viscous fluid (see
Fig.~\ref{fig:cmame-ball-geometry}). The referential mass density of
the disk is less than that of the surrounding fluid,
i.e. $\rho_{\s_{0}}< \rho_{\f}$. The bottom and the sides of the
control volume have homogeneous Dirichlet boundary condition, while
the top side has homogeneous Neumann boundary conditions. This ensures
that fluid can freely enter and exit the control volume along the top
edge. As the disk tries to recover its undeformed state, it expands
and causes the flux to exit the control volume. Thus the change in the
area of the disk from its initial state, over a certain interval of
time, matches the amount of fluid efflux from the control volume over
the same interval of time. We use this idea to estimate the error in
our numerical method.

In this test, the following parameters have been used:
$R=\np[m]{0.125}$, $l=\np[m]{1.0}$,
$\rho_{\s_{0}}=\np[kg/m^{3}]{0.8}$, $\rho_{\f}=\np[kg/m^{3}]{1.0}$,
$\mu^{e}=\np[Pa]{20}$, $\mu_{\s}= \np[Pa \! \cdot \! s]{2.0}$,
$\mu_{\f}= \np[Pa \! \cdot \! s]{0.01}$, $\lambda_{0}=0.7$. The body
force on the system is $\b=\np[m/s^{2}]{(0,-10)}$. The initial
location of the center of the disk is
$\x_{C0}=\np[m]{(0.6, 0.4)}$. We have used $Q2-P1$ elements for
the control volume and the mesh comprises 1024 cells and 11522
DoFs. $Q2$ elements have been used for disk and the mesh comprises 224
cells with 1894 DoFs.

In
Figs.~\ref{subfig:cmame-ball-visit-dt0}--\ref{subfig:cmame-ball-visit-dt300},
we can see the velocity field over the entire control volume due to
the motion of the disk as well as the pressure in the fluid for
several instants of time spanning the duration of the simulation. The
initial deformation of the disk is such that the density of the disk
is greater than that of the fluid. But as soon as the disk is
released, it starts to expand while remaining at almost the same
vertical height, instead of descending, as can be seen from
Fig.~\ref{fig:cmame-ball-vertical-posn}. This causes a surge of fluid
outflow that as shown in Fig.~\ref{fig:cmame-ball-flux}. The expansion
of the disk results in an increased buoyancy on the disk that begins
to rise through the fluid. As the solid rises the hydrostatic pressure
from the fluid decreases and hence it grows further (see
Fig.~\ref{fig:cmame-ball-volume-change}) till it reaches the top of
the domain. The amount of fluid ejected from the control volume due to
expansion of the disk does not quite keep pace with this expansion
(see Fig.~\ref{fig:cmame-ball-volume-change}). The difference in these
two amounts is a measure of the error in our numerical method (see
Fig.~\ref{fig:cmame-ball-volume-exchange-error}).

%%%%%%%%%%%%%%%%%%%%%%%%%%%%%%%%%%%%%%%%%%%%%%%%%%%%%%%%%%
\begin{figure}[htbp]%%------------Figure start ------------%%
	\begin{center}
	\subfigure[$t={\np[s]{0}}$]
	{\includegraphics[width=0.48\textwidth]
	{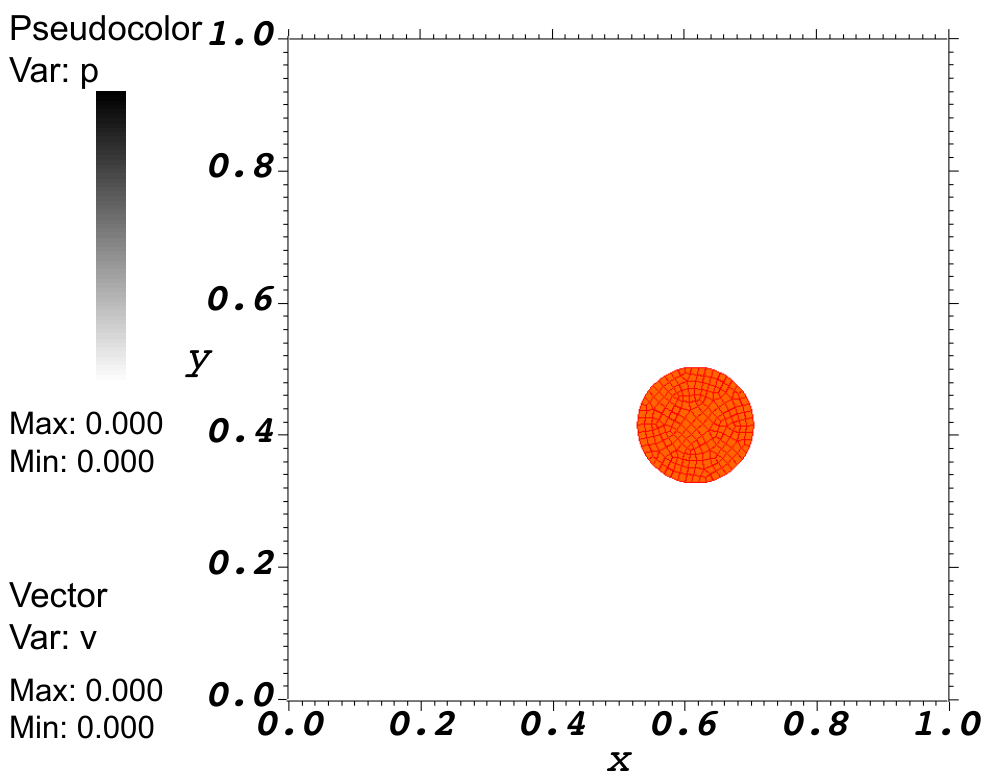}
	\label{subfig:cmame-ball-visit-dt0}
	}
	\subfigure[$t={\np[s]{0.01}}$]
	{\includegraphics[width=0.48\textwidth]
	{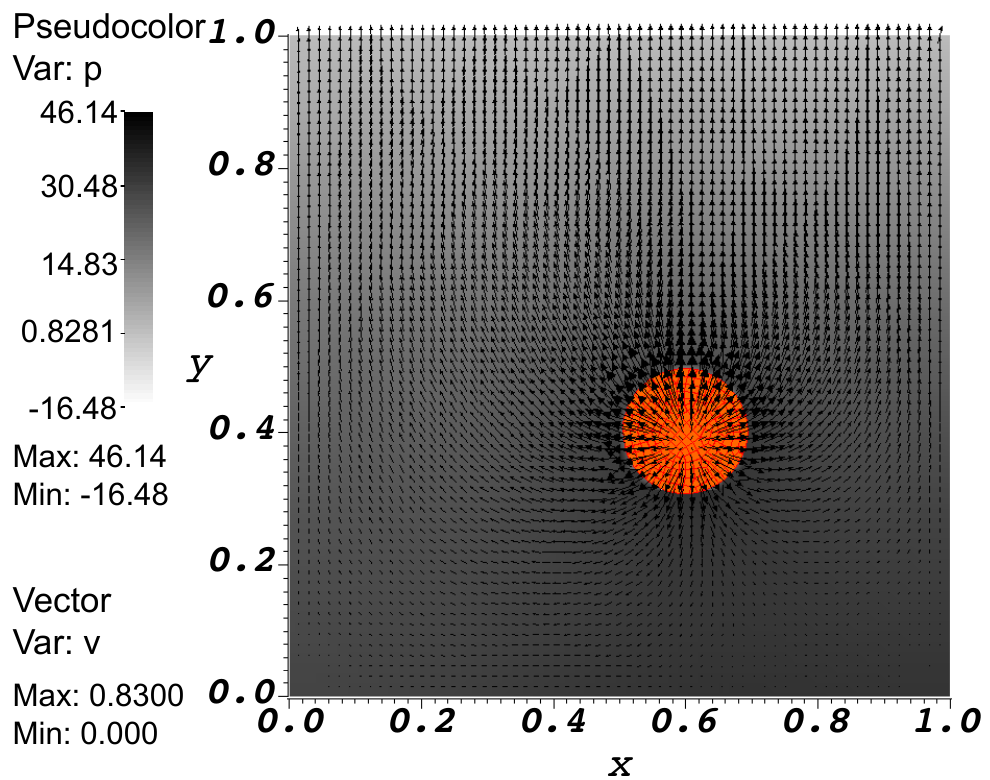}
	\label{subfig:cmame-ball-visit-dt1}
	}
	\subfigure[$t={\np[s]{0.5}}$]
	{\includegraphics[width=0.48\textwidth]
	{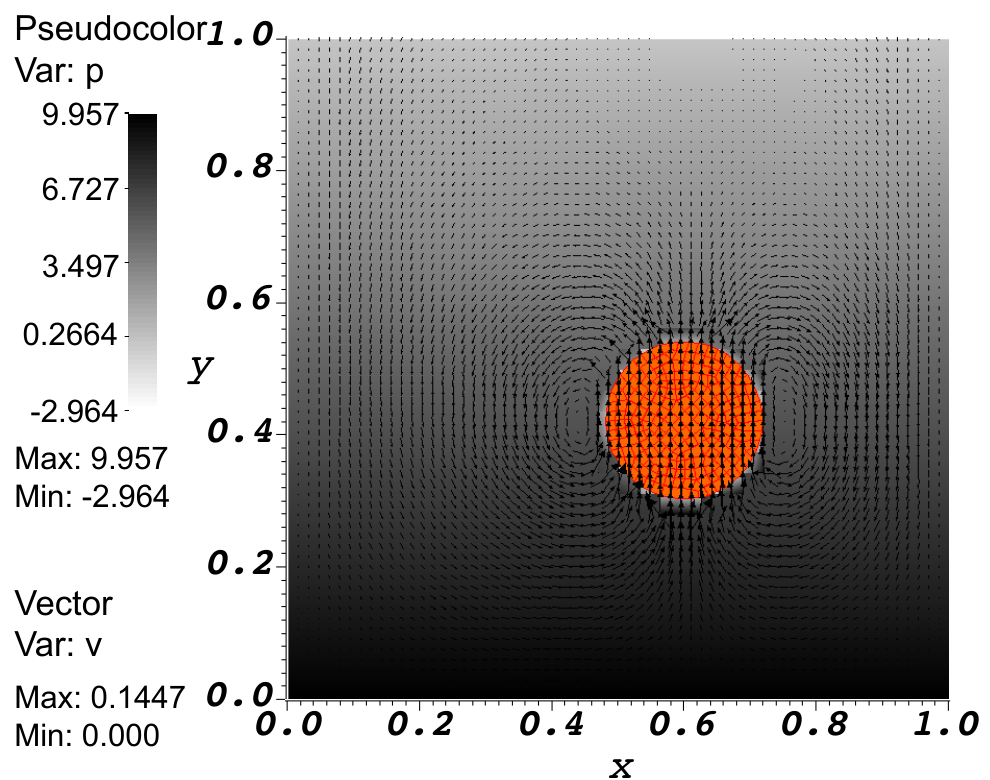}
	\label{subfig:cmame-ball-visit-dt50}
	}
	\subfigure[$t={\np[s]{1.0}}$]
	{\includegraphics[width=0.48\textwidth]
	{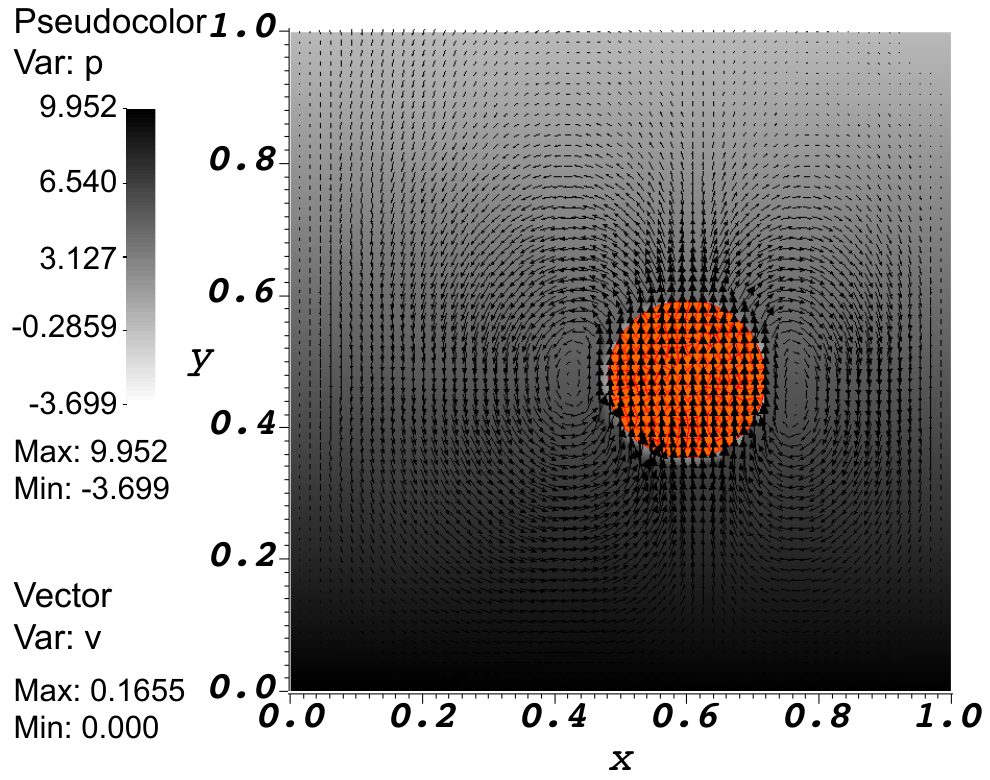}
	\label{subfig:cmame-ball-visit-dt100}
	}
	\subfigure[$t={\np[s]{2.0}}$]
	{\includegraphics[width=0.48\textwidth]
	{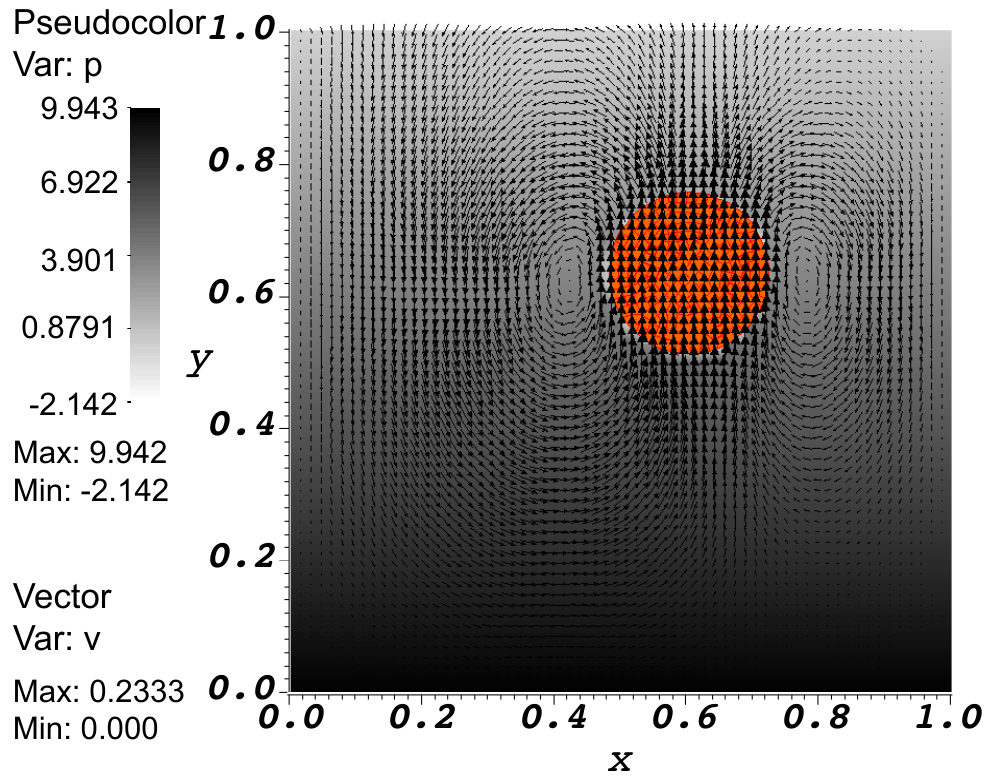}
	\label{subfig:cmame-ball-visit-dt200}
	}
	\subfigure[$t={\np[s]{3.0}}$]
	{\includegraphics[width=0.48\textwidth]
	{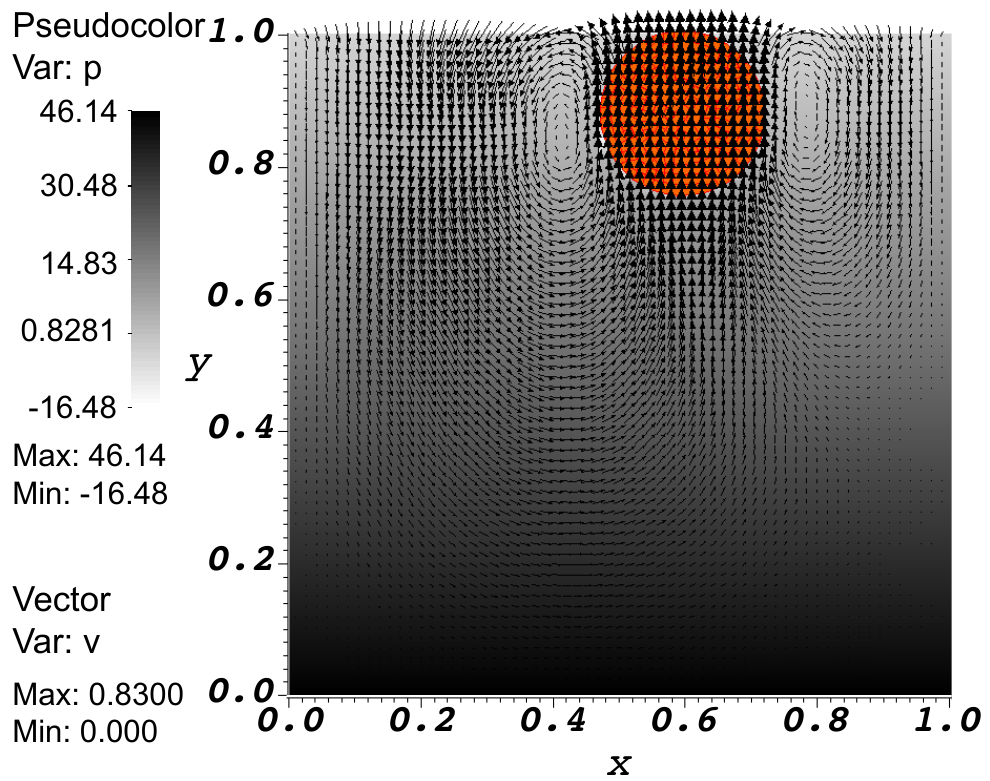}
	\label{subfig:cmame-ball-visit-dt300}
	}
	\caption{The velocity field over the entire control volume, the pressure in the fluid and the mesh of the disk.}
	\label{fig:cmame-ball-visitplots}
	\end{center}
\end{figure}%%------------Figure end ------------%%

\begin{figure}[htbp]%%------------Figure start ------------%%
\begin{center}
\subfigure
{\includegraphics[width=0.48\textwidth, trim=2.2in 1in 0.9in 0.8in, clip=true]{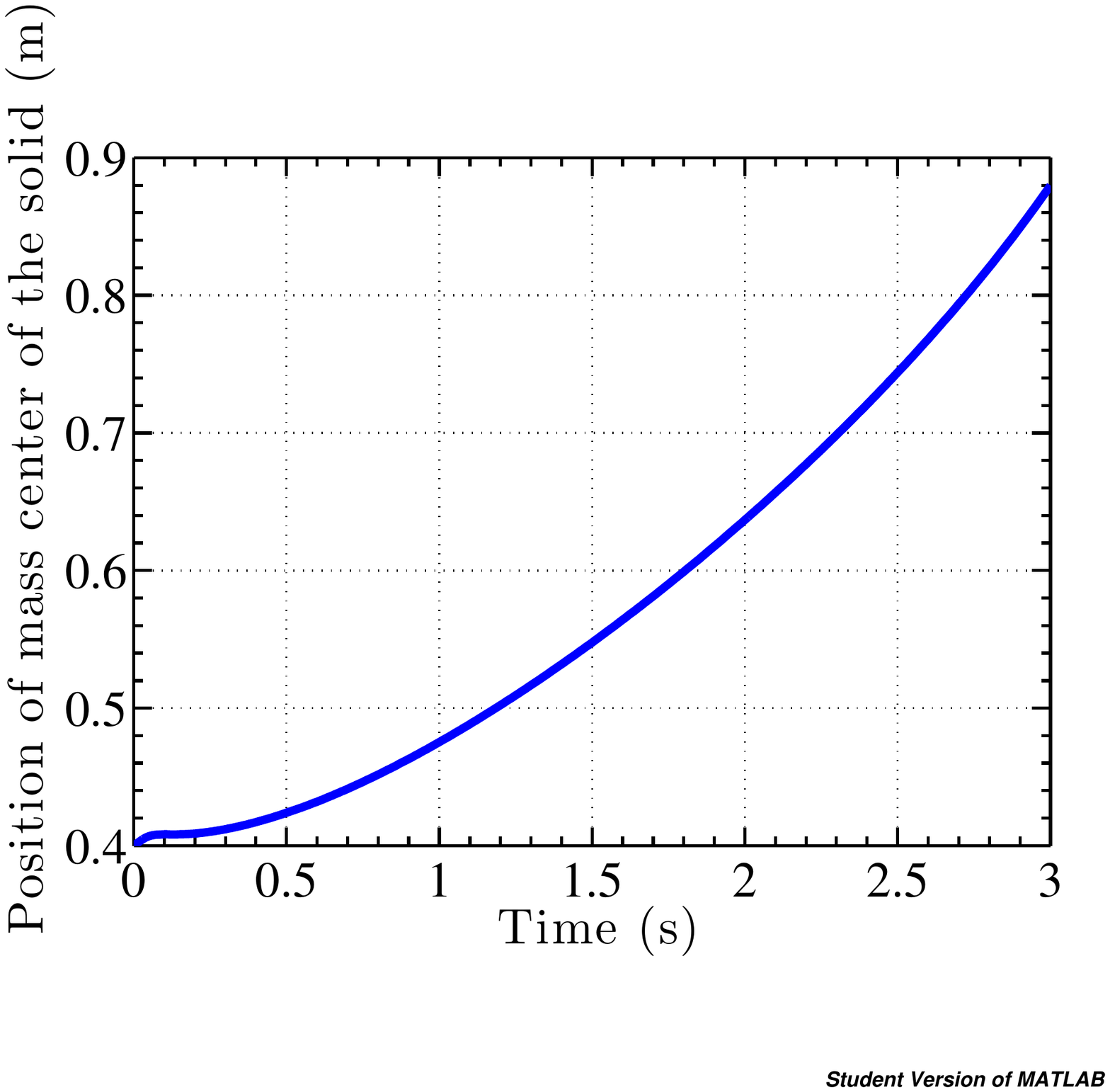}
\label{fig:cmame-ball-vertical-posn}
}
\subfigure
{\includegraphics[width=0.48\textwidth, trim=1in 2.8in 0.9in 3.2in, clip=true]{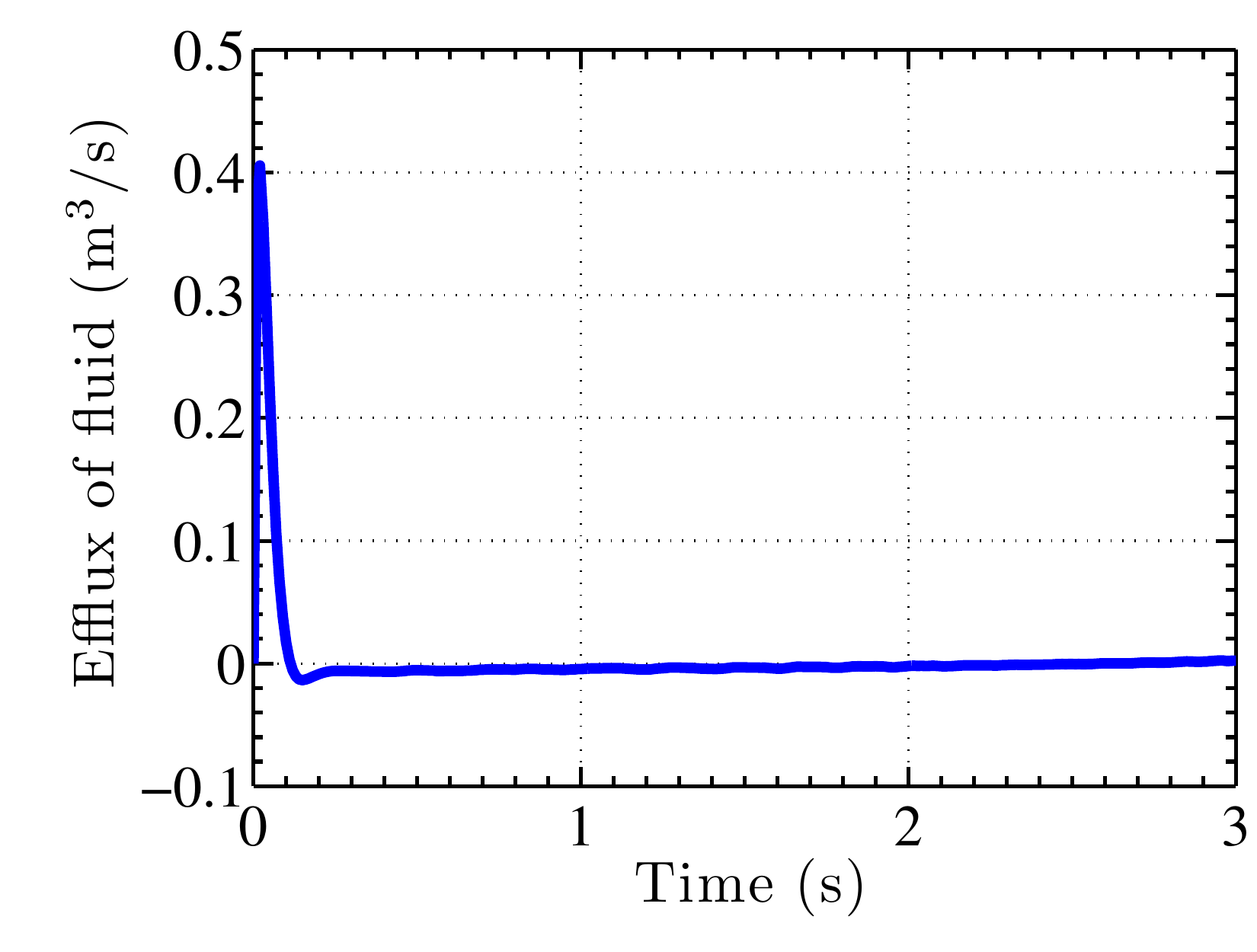}
\label{fig:cmame-ball-flux}
}
\caption{Instantaneous vertical position of the disk and flux of the fluid.}
\label{fig:cmame-ball-verticalposn-flux}
\end{center}
\end{figure}%%------------Figure end ------------%%

\begin{figure}[htbp]%%------------Figure start ------------%%
\begin{center}
\subfigure
{\includegraphics[width=0.47\textwidth, trim=0.8in 2.8in 0.9in 3.2in, clip=true]{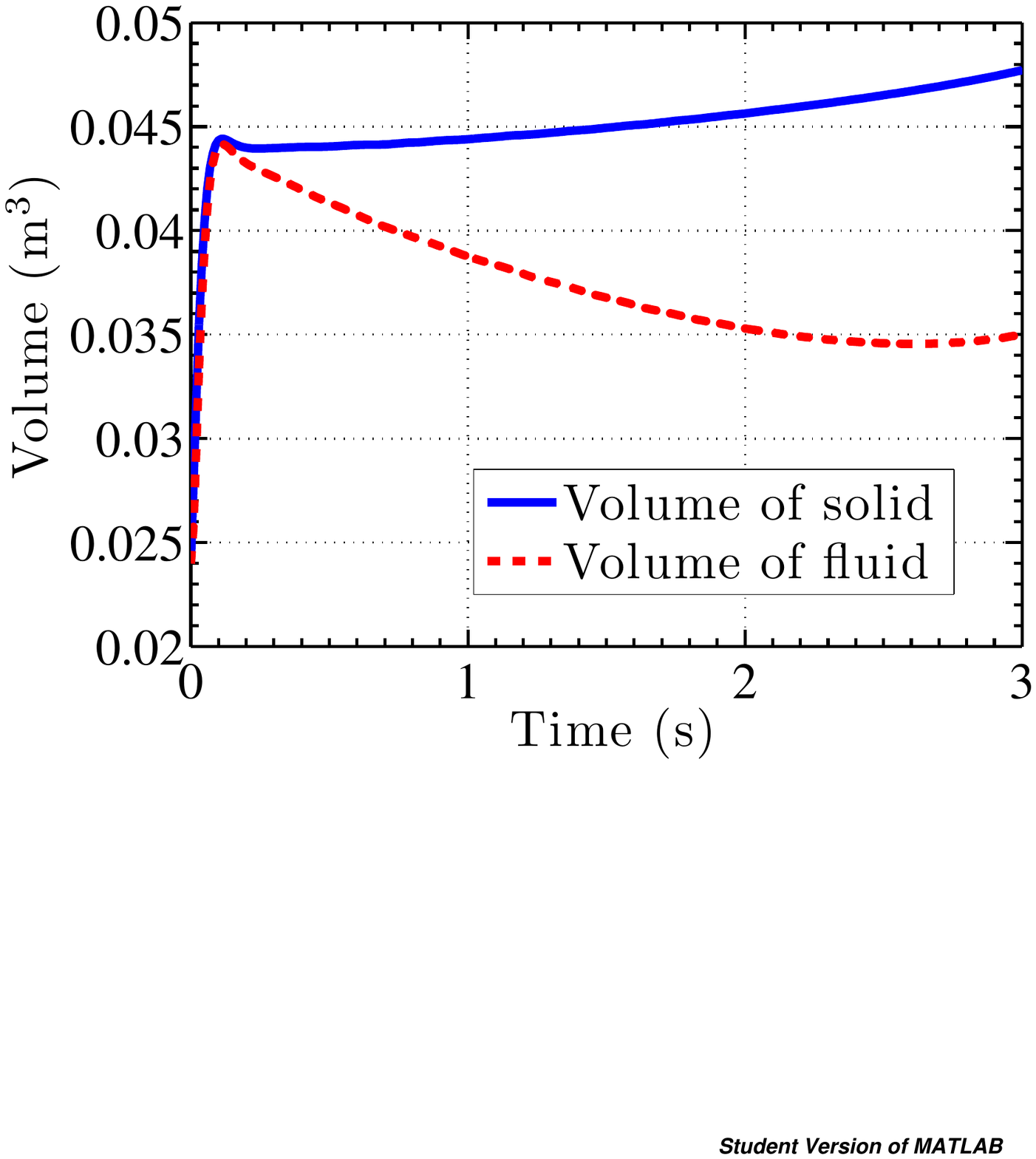}
\label{fig:cmame-ball-volume-change}
}
\subfigure
{\includegraphics[width=0.48\textwidth, trim=0.1in 2.6in 0.5in 2.9in, clip=true]{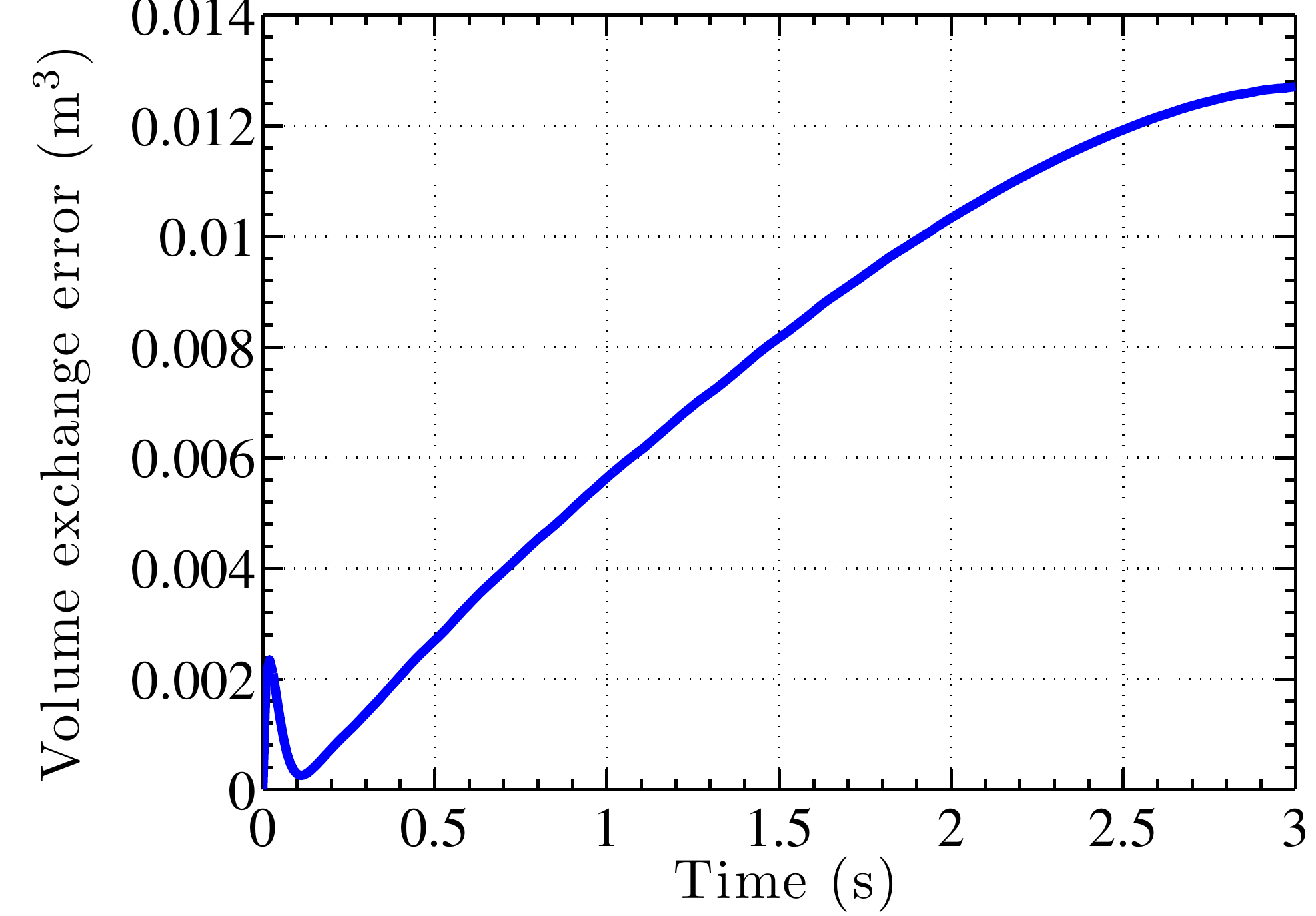} %0.3in 0.6in 0.9in 0.8in,
\label{fig:cmame-ball-volume-exchange-error}
}
\caption{Instantaneous area change of the disk, amount of fluid ejected from the control volume and the difference between them as an estimate of the error in our numerical implementation.}
\label{fig:cmame-ball-error-estimate}
\end{center}
\end{figure}%%------------Figure end ------------%%

%%%%%%%%%%%%%%%%%%%%%%%%%%%%%%%%%%%%%%%%%%%%
%%%---------------------------------SECTION----------------------------------%%%
%%%%%%%%%%%%%%%%%%%%%%%%%%%%%%%%%%%%%%%%%%%%
\subsection{Deformation of a compressible annulus under the action of point source of fluid}

\begin{figure}[htbp]%%------------Figure start ------------%%
	\begin{center}
		\includegraphics{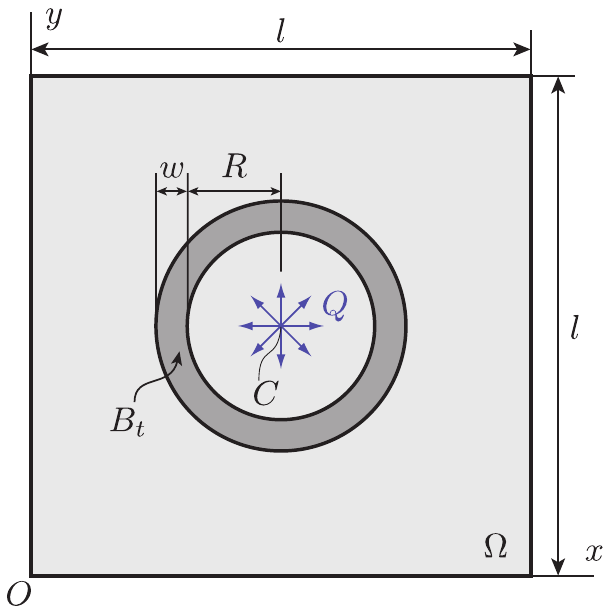}
		\caption{Initial configuration of an annulus immersed
                  in a square box filled with fluid. At the center $C$
                  of the box there is a point source of strength $Q$.}
		\label{fig:PointSource-Annulus-Geometry}
	\end{center}
\end{figure}%%------------Figure end ------------%%

This test was first proposed by Roy~\cite{Roy-2012-a}, as a toy model
to describe the behavior of hydrocephalus in the brain. In this test
we observe the deformation of a hollow cylinder, submerged in a fluid
contained in a rigid prismatic box, due to the influx of fluid along
the axis of the cylinder. In the two-dimensional context the test
comprises an annular solid with inner radius $R$ and thickness $w$
that is concentric with a fluid-filled square box of edge length $l$
(see, Fig.~\ref{fig:PointSource-Annulus-Geometry}). A point mass
source of fluid of strength $Q$ is located at $\x_{C}$ which
corresponds to the center $C$ of the control volume.  The radially
symmetric nature of point source ensures that momentum balance law
remains unaltered. However, we need to modify the balance of mass to
account for the mass influx from the point source by adding the
following term to the pressure equation:
%%================Equation start=================%%
\begin{equation} \label{eqn:WeakBalOfMass-PtSource}
\dfrac{Q}{\rho_{\f}}\delta \left(\x-\x_{C}\right) 
\end{equation}
%%================Equation end=================%%
For this test we have used a single point source whose strength is a
constant, and all boundary conditions on the control volume are of
homogeneous Dirichlet type.The solid is compressible and hence volume
of solid and thereby the volume of fluid in the control volume can
change. The homogeneous Dirichlet boundary condition implies that the
fluid cannot leave the control volume and hence the amount of fluid
that accumulates in the control volume due to the point source must
equate the decrease in volume of the solid. The difference in these
two volumes can serve as an estimate of the numerical error incurred.
\begin{table}[htbp]\small%%--------Table begin---------------%%
\caption{Number of cells and DoFs used in the different simulations involving the deformation of a compressible annulus under the action of a point source.}
\begin{center}
\begin{tabular*} {\textwidth} {@{\extracolsep{\fill}} c c c c c}
\toprule
&
\multicolumn{2}{c}{Solid}&
\multicolumn{2}{c}{Control Volume} 
\\
\cmidrule(r){2-3}\cmidrule(r){4-5}
&
Cells &
DoFs &
Cells &
DoFs
\\
\midrule
Level 1&  6240&   50960&  1024&   9539 \\
Level 2& 24960&  201760&  4096&  37507 \\
Level 3& 99840&  802880& 16384& 148739 \\
\bottomrule
\end{tabular*}
\end{center}
\label{tab:mesh-refinement-pts-comp-solid}
\end{table}%%---------Table end--------------%%

We have used the following parameters for this test: $R =\np[m]{0.25}$, $w=\np[m]{0.05}$, $l=\np[m]{1.0}$, $\rho_{\f}= \rho_{\s_{0}}= \np[kg/m^{3}]{1}$, $\mu_{\f}= \mu_{\s}= \np[Pa \! \cdot \! s]{1}$, $\mu^{e}=\np[Pa]{1}$, $\nu=0.3$, $Q=\np[kg/s]{0.1}$ and $dt=\np[s]{0.01}$. We have tested for three different mesh refinement levels whose details have been listed in Table~\ref{tab:mesh-refinement-pts-comp-solid}.
\begin{figure}[htbp]%%------------Figure start ------------%%
\begin{center}
\subfigure[$t={\np[s]{0}}$]
{\includegraphics[width=0.48\textwidth]{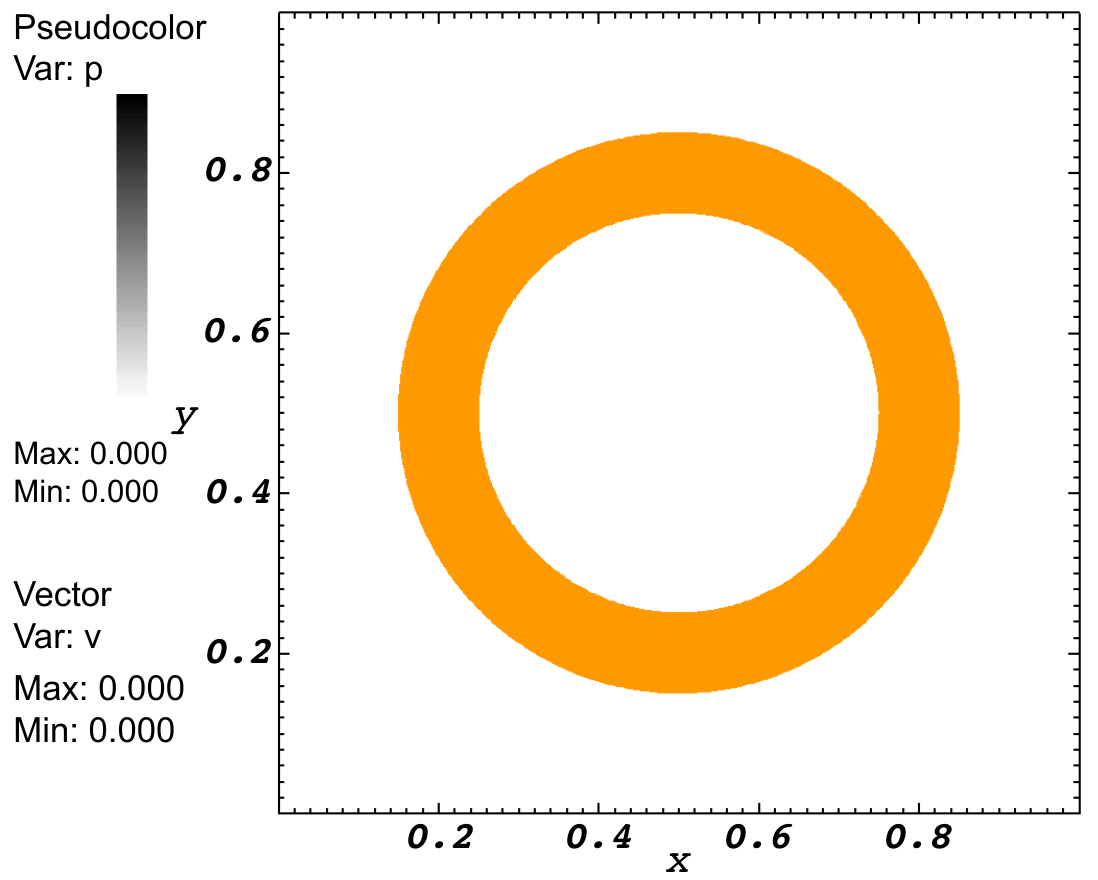}
\label{subfig:pts-comp-solid-t0}
}
\subfigure[$t={\np[s]{1.0}}$]
{\includegraphics[width=0.48\textwidth]{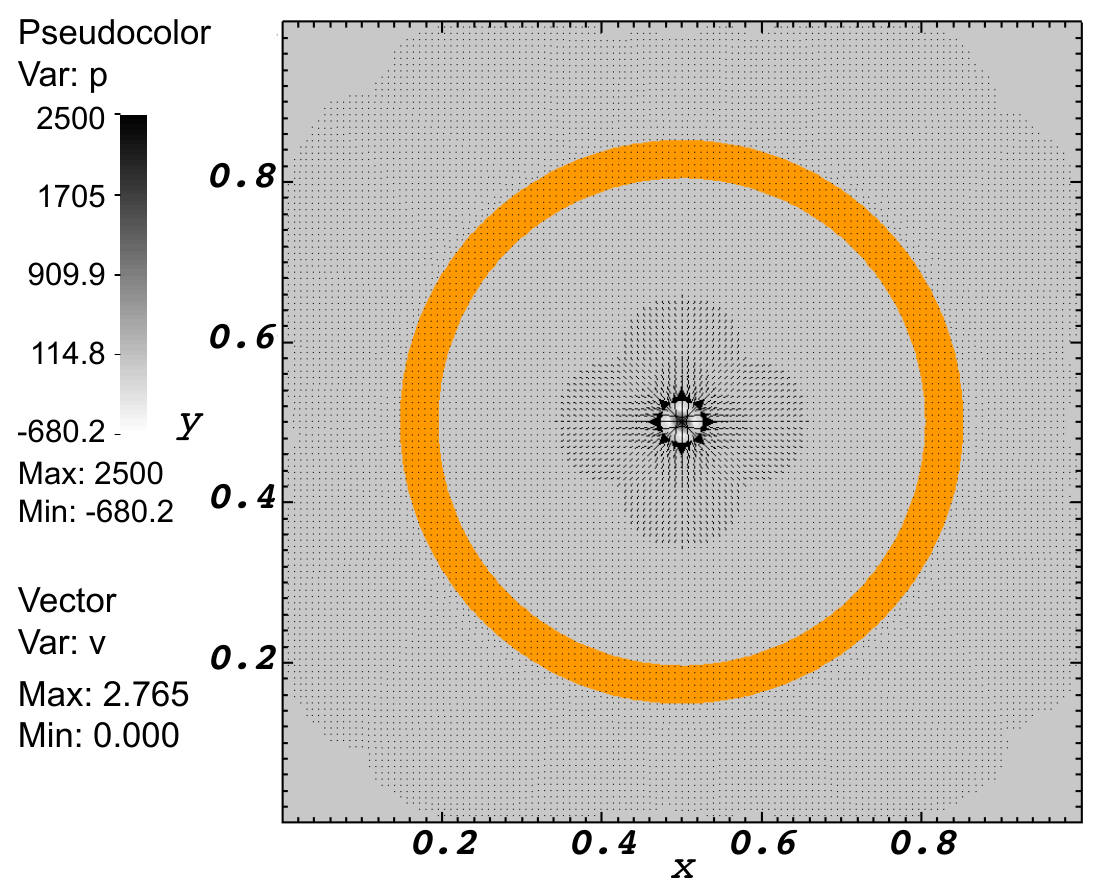}
\label{subfig:pts-comp-solid-t1}
}
\caption{The velocity and the mean normal stress field over the control volume. Also shown is the annulus mesh.}
\label{default}
\end{center}
\end{figure}%%------------Figure end ------------%%

The initial state of the system is shown in Fig.~\ref{subfig:pts-comp-solid-t0}. As time progresses, the fluid entering the control volume deforms and compresses the annulus as shown in Fig.~\ref{subfig:pts-comp-solid-t1} for $t=\np[s]{1}$. When we look at the difference in the instantaneous amount of fluid entering the control volume and the decrease in the volume of the solid, we see that the difference increases over time (see Fig.~\ref{fig:pts-comp-solid-error-estimate}). This is not surprising since the mesh of the solid becomes progressively distorted as the fluid emanating from the point source push the inner boundary of the annulus. The error significantly reduces with the increase in the refinements of the fluid and the solid meshes. 
\begin{figure}[htbp]%%------------Figure start ------------%%
\begin{center}
\subfigure
{\includegraphics[width=0.48\textwidth, trim = 0.5in 2.6in 0.9in 2.5in, clip=true]{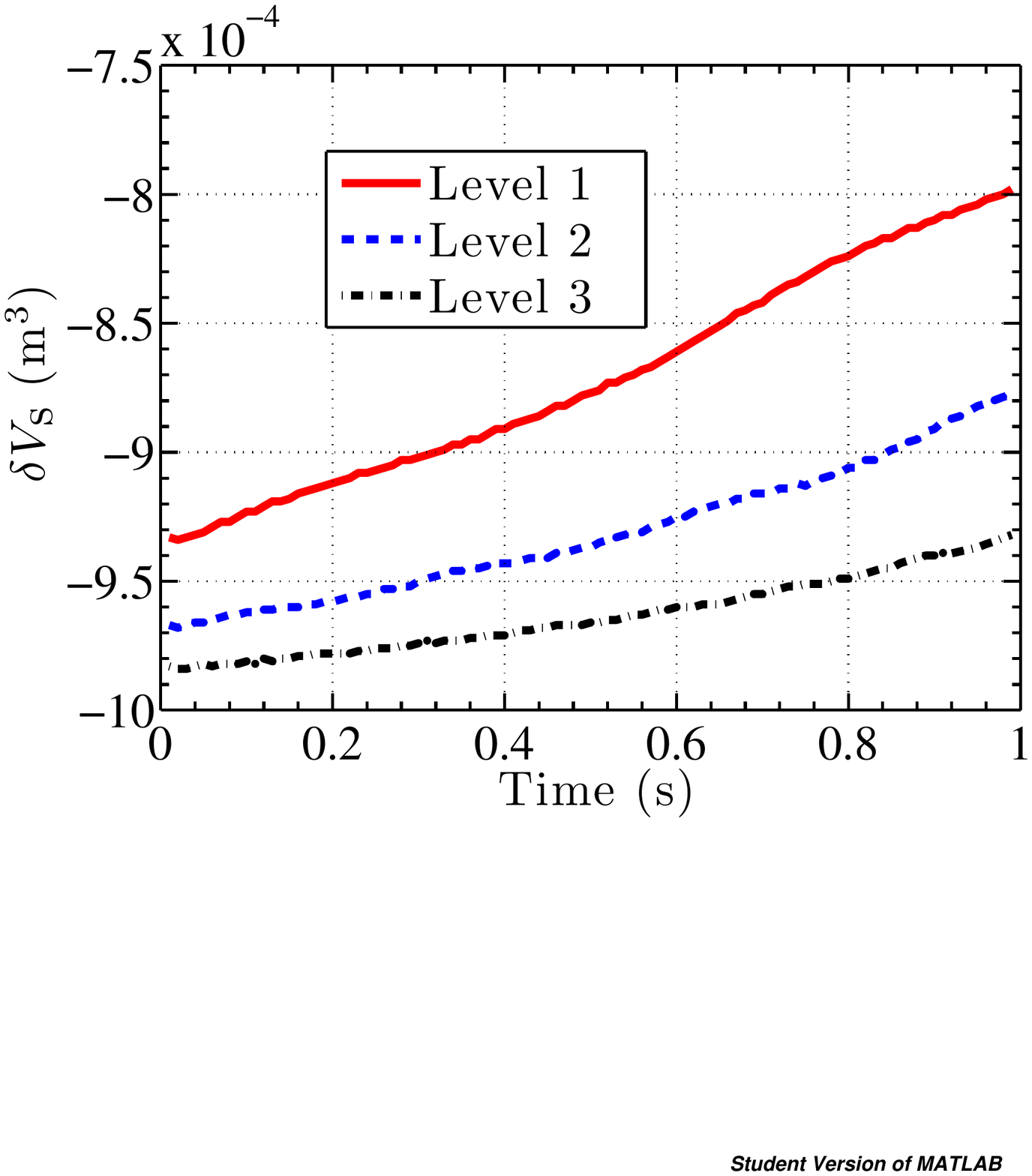}
\label{subfig:pts-comp-solid-error-estimate-1}
}
\subfigure
{\includegraphics[width=0.47\textwidth, trim = 0.6in 2.6in 0.9in 2.5in, clip=true]{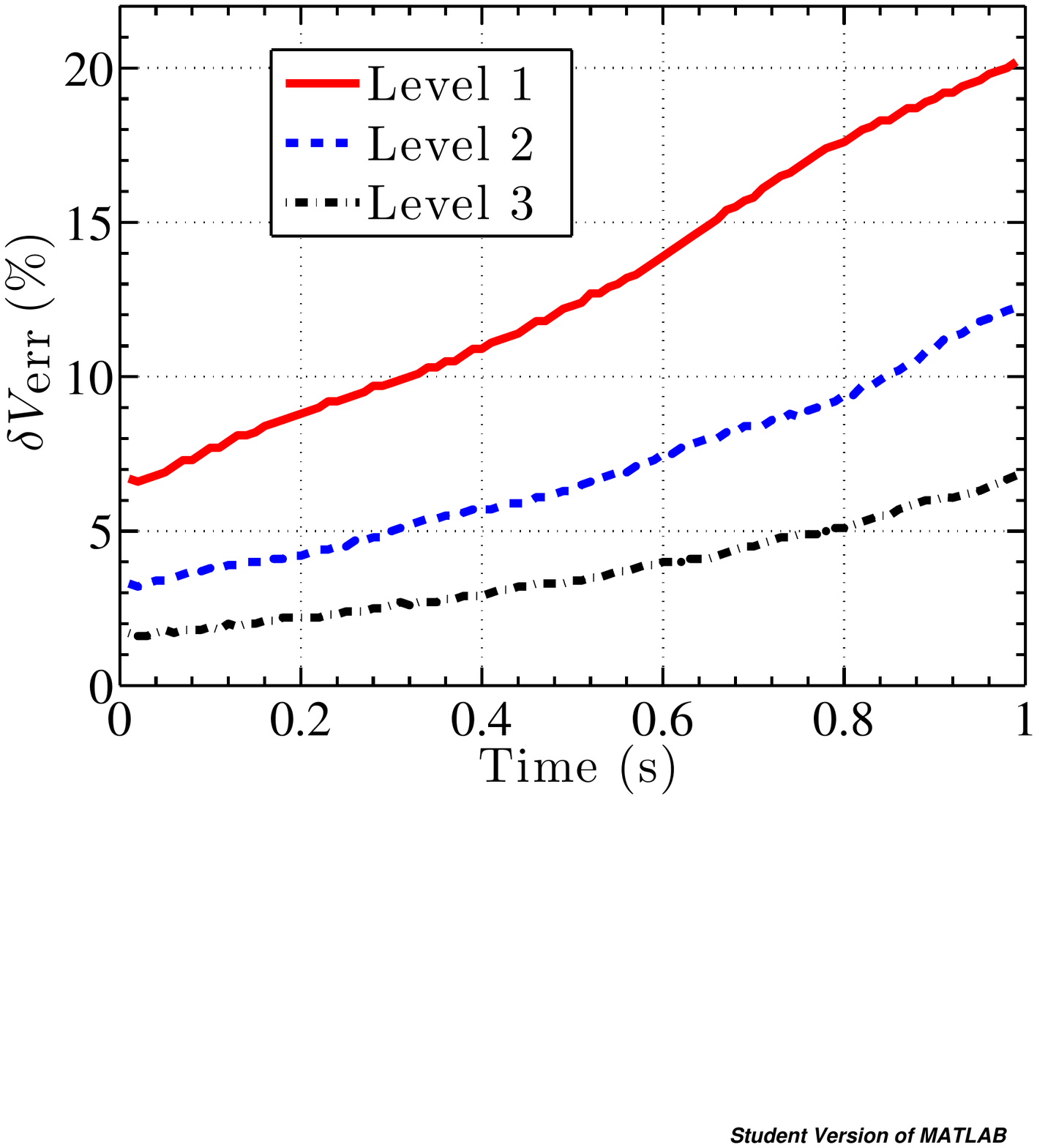}
\label{subfig:pts-comp-solid-error-estimate-2}
}
\caption{The difference between the instantaneous amount of fluid entering due to the source and the change in the area of the annulus. The difference reduces with mesh refinement.}
\label{fig:pts-comp-solid-error-estimate}
\end{center}
\end{figure}%%------------Figure end ------------%%

\section{Conclusions}
\label{sec:conclusions}

The Finite Element Immersed Boundary Method (FEIBM) is a well
established formulation for fluid structure interaction problems of
general types. In most implementations (see, for example,
~\cite{BoffiGastaldiHeltaiPeskin-2008-a}), the solid constitutive
behavior is constrained to be viscous (with the same viscosity of the
surrounding fluid) and incompressible.

The first attempt to allow for solids of arbitrary constitutive type
was presented in~\cite{HeltaiCostanzo-2012-a}, whose formulation is
applicable to problems with immersed bodies of general topological and
constitutive characteristics. Such formulation, however, did not
expose the intrinsic structure of the underlying problem. 

In this work we have shows how the incompressible version of the FSI
model presented in~\cite{HeltaiCostanzo-2012-a} can be seen as a
special case of the Distributed Lagrange Multiplier method, introduced
in \cite{BoffiCavalliniGastaldi-2015-a}, and we presented a novel
distributed Lagrange multiplier method that generalizes the
compressible model introduced in~\cite{HeltaiCostanzo-2012-a}.

Two validation tests are presented, to demonstrate the capability of
the model to take into account complex fluid structure interaction
problems between compressible solids and incompressible fluids.

\section*{Acknowledgments}

This work has been partly supported by IMATI/CNR and GNCS/INDAM.

\bibliographystyle{plain}
\bibliography{ref}

\end{document}